\documentclass[10pt,twoside]{article}
\usepackage{mathrsfs}
\usepackage{amsmath}
\usepackage{amssymb}
\usepackage{fancyhdr}
\usepackage{latexsym}
\usepackage{bbding}
\usepackage{mathrsfs}
\usepackage{wasysym}
\usepackage{multicol,graphics}
\usepackage{cite}

\newtheorem{theorem}{Theorem}[section]
\newtheorem{lemma}[theorem]{Lemma}
\newtheorem{corollary}[theorem]{Corollary}

\newtheorem{definition}[theorem]{Definition}

\newtheorem{proposition}[theorem]{Proposition}
\numberwithin{equation}{section}
\newenvironment{proof}[1][Proof]{\noindent\textbf{#1.} }{\hfill $\Box$}
\allowdisplaybreaks

\makeatletter
\begin{document}
\title{{Global large solutions of the Cauchy problem for the NS-NPP equations with Fujita-Kato type initial data}\footnote{Email address: jihzhao@163.com (J. Zhao), \; chaodeng03@163.com (C. Deng)}}
\author{Jihong Zhao\\
[0.2cm] {\small School of Mathematics and Information Science, Baoji University of Arts and Sciences,}\\
[0.2cm] {\small  Baoji, Shaanxi 721013,  China} \\  \\
Chao Deng$^\dag$\\
[0.2cm] {\small School of Mathematics and Statistics, Jiangsu Normal University,}\\
[0.2cm] {\small  Xuzhou,  Jiangsu 221116,  China} }
\date{}
\maketitle

\begin{abstract}
In this paper, we prove the global well-posedness of the Cauchy problem for the NS-NPP equations with some large Fujita-Kato type initial data. Specifically, we show that there
exist two positive constants $c_{0}$ and $C_{0}$ such that if the initial data $(u_{0}, N_{0}, P_{0})$ satisfies the following condition:
\begin{equation*}
   \left(\|u_0\|_{\dot{H}^{-1+\frac{d}{2}}}+\|N_{0}-P_{0}\|_{\dot{H}^{-2+\frac{d}{2}}}\right)\exp\left\{C_{0}\big(\|N_{0}+P_{0}\|_{\dot{H}^{-2+\frac{d}{2}}}^2+1\big)\right\}
  \leq
  c_{0},
\end{equation*}
then the NS-NPP equations admits a unique global solution. This result implies global existence of solutions without any smallness conditions imposed on the sum of initial particle densities of negative and positive electric charge in the framework of Sobolev spaces.

\textbf{Keywords}: NS-NPP equations; large solutions; global existence; Sobolev spaces.

\textbf{2020 AMS Subject Classification}: 35K15, 35K55, 35Q35, 76A05
\end{abstract}

\section{Introduction}
In this paper, we study the Cauchy problem of the following incompressible Navier-Stokes-Nernst-Planck-Poisson (NS-NPP) equations arising from electro-hydrodynamics (cf. \cite{R90, S09}):
\begin{equation}\label{NSNPP}
\begin{cases}
  \partial_{t} u+u\cdot\nabla u-\Delta
  u+\nabla \pi=\Delta
  \phi\nabla\phi,\\
  \nabla\cdot u=0,\\
  \partial_{t} N+u\cdot \nabla
  N=\nabla\cdot(\nabla N-N\nabla \phi),\\
  \partial_{t} P+u\cdot \nabla
  P=\nabla\cdot(\nabla P+P\nabla \phi),\\
 -\Delta \phi=P-N,\\
 u(x,0)=u_0(x), N(x,0)=N_0(x), P(x,0)=P_0(x),
\end{cases}
\end{equation}
where the unknowns $u$, $\pi$, $N (P)$  and $\phi$ stand for the fluid velocity, the pressure, the particle density of  negative (positive) electric charge and the electrostatic potential, respectively. Without loss of generality,  we have assumed that the fluid density, viscosity, charge mobility and dielectric constant are unity.
\smallbreak

Notice that the first two equations of \eqref{NSNPP} are the momentum
conservation and the mass conservation equations of the
incompressible flow with an external force $\Delta
  \phi\nabla\phi$. In the case that the flow is charge-free, the
system \eqref{NSNPP} reduces into the well-known incompressible Navier-Stokes (NS) equations:
\begin{equation}\label{NS}
\begin{cases}
 \partial_{t} u-\Delta u+u\cdot\nabla u+\nabla \pi=0,\\
  \nabla\cdot u=0,
\\
  u(x,0)=u_0(x).
\end{cases}
\end{equation}
It is easy to see that if the pair $(u(x,t), \pi(x,t))$ solves the NS equations \eqref{NS},  then for all $\lambda>0$, the pair $(u_{\lambda}(x,t), \pi_{\lambda}(x,t))$ with
$$
  u_{\lambda}(x,t):=\lambda u(\lambda x, \lambda^2t),\ \
  \pi_{\lambda}(x,t):=\lambda^2 \pi(\lambda x, \lambda^2t)
$$
is also a solution to the NS equations \eqref{NS} with the initial data
\begin{equation}\label{scaling invariance of NS}
 u_{0,\lambda}(x):=\lambda u_0(\lambda x).
\end{equation}
The spaces which are invariant under the scaling \eqref{scaling invariance of NS} are called critical spaces for the NS equations \eqref{NS}. Examples of critical spaces for the NS equations are \begin{equation*}
 \dot{H}^{-1+\frac{d}{2}}\hookrightarrow L^d\hookrightarrow\dot{B}^{-1+\frac{d}{p}}_{p,r}\hookrightarrow BMO^{-1}\hookrightarrow \dot{B}^{-1}_{\infty,\infty}\ \ \text{with}
 \ \ 2 \le d <p<\infty,\ \ 2\leq r\leq \infty.
\end{equation*}
 The well-posed result in $\dot{H}^{-1+\frac{d}{2}}$ is due to  Fujita-Kato \cite{FK64} (for the 3-dimensional case see also \cite{L34}, where the smallness of $u_{0}$ is measured by $\|u_{0}\|_{L^2}\|\nabla u_{0}\|_{L^2}$). Since then, a number of works have been devoted to establishing similar well-posed results with initial data in critical spaces. Kato \cite{K84} proved the global well-posedness for small
initial data and local well-posedness for large initial data in the critical Lebesgue space $L^{d}$ (see also \cite{G86}). Later,  Cannone-Planchon \cite{C97, CP96} proved the same results for initial data in the critical Besov spaces $\dot{B}^{-1+\frac{d}{p}}_{p,\infty}$ (see also the different proof by Chemin \cite{C92}), and Koch-Tataru \cite{KT01} for initial data in $BMO^{-1}$. For the largest scaling invariant space $\dot{B}^{-1}_{\infty,\infty}$,  Bourgain-Pavlovi\'{c} \cite{BP08} proved ill-posedness for the NS equations \eqref{NS}. We refer the readers to see Yoneda \cite{Y10} and Wang \cite{W15} for more ill-posedness results.
\smallbreak

On the other hand, the third and the fourth equations of  \eqref{NSNPP} are the Nernst-Planck equations modified by the convective terms $u\cdot \nabla N$ and $u\cdot \nabla P$, and they are coupled by the Poisson equation (the fifth equation of  \eqref{NSNPP}). In the absence of the fluid, the equations \eqref{NSNPP} reduces to the Nernst-Planck-Poisson (NPP) equations in semi-conductor devices:
\begin{equation}\label{NPP}
\begin{cases}
  \partial_{t}N-\Delta N=-\nabla\cdot(N\nabla \phi),\\
   \partial_{t}P-\Delta P=\nabla\cdot(P\nabla \phi),\\
  -\Delta \phi=P-N,\\
  N(x,0)=N_0(x), \ \ P(x,0)=P_0(x).
\end{cases}
\end{equation}
The NPP equations \eqref{NPP} is a basic model for the diffusion of ions in an electrolyte  (cf. \cite{DH23}). It is easy to see that if the pair $(N(x,t), P(x,t),\phi(x,t))$ solves the NPP equations \eqref{NPP},  then for all $\lambda>0$, the pair $(N_{\lambda}(x,t), P_{\lambda}(x,t),\phi_{\lambda}(x,t))$ with
$$
  N_{\lambda}(x,t):=\lambda^2 N(\lambda x, \lambda^2t),\ \
  P_{\lambda}(x,t):=\lambda^2 P(\lambda x, \lambda^2t), \ \
  \phi_{\lambda}(x,t):=\phi(\lambda x, \lambda^2t)
$$
is also a solution to the NPP equations \eqref{NPP} with the initial data
\begin{equation}\label{scaling invariance of NPP}
  N_{0,\lambda}(x):=\lambda^2 N_0(\lambda x),\ \
  P_{0,\lambda}(x):=\lambda^2 P_0(\lambda x).
\end{equation}
The spaces which are invariant under the scaling \eqref{scaling invariance of NPP} are called critical spaces for the NPP equations \eqref{NPP}. Examples of critical spaces are the Hardy space $\mathcal{H}^{1}$ in 2-dimensional case, and the Sobolev space $\dot{H}^{-2+\frac{d}{2}}$, the Lebesgue space $L^\frac{d}{2}$ and the Besov space $\dot{B}^{-2+\frac{d}{p}}_{p,r}$ in general d-dimensional case. The study of the NPP equations \eqref{NPP} in critical spaces was also well-developed.  The result in $ L^{\frac{d}{2}}$ is due to Kurokiba-Ogawa \cite{KO08}. Ogawa-Shimizu \cite{OS08, OS10} established global existence of solutions for 2-dimensional equations \eqref{NPP} with small initial data in critical Hardy space $\mathcal{H}^{1}$ and the homogeneous Besov space $\dot{B}^{0}_{1,2}$, respectively. Karch \cite{K99} proved global well-posedness for small initial data in critical Besov space $\dot{B}^{-2+\frac{d}{p}}_{p,\infty}$ for $\frac{d}{2}<p<d$,  and Zhao-Liu-Cui \cite{ZLC13} further extended to $\dot{B}^{-2+\frac{d}{p}}_{p, r}$ for $1<p<2d$ and $1\leq r\leq \infty$.  Deng-Li \cite{DL13} proved that, in the 2-dimensional case,  the equations \eqref{NPP} is well-posed in critical space  $\dot{B}^{-\frac{3}{2}}_{4,2}$ and ill-posed in critical space $\dot{B}^{-\frac{3}{2}}_{4,q}$ for $2<q\leq \infty$, and this ill-posedness result was further extended to the general d-dimensional case by  Iwabuchi-Ogawa \cite{IO16}, where they proved that the equations \eqref{NPP} is ill-posed in  $\dot{B}^{-2+\frac{d}{p}}_{p,q}$ with $2d<p\leq \infty$ and $1\leq q\leq \infty$, or $p=2d$ and $2<q\leq \infty$.
\smallbreak

Recently, in \cite{ZL23},  the first author of this paper and Liu proved that there exist two constants $c_{0}$ and $C_{0}$ such that if the initial data $(N_{0}, P_{0})\in \dot{B}^{-2+\frac{d}{p}}_{p,1}$ and satisfies the following condition:
\begin{equation}\label{initial condition of npp}
 \|N_{0}-P_{0}\|_{\dot{B}^{-2+\frac{d}{p}}_{p,1}}\exp\{C_{0}\|N_{0}+P_{0}\|_{\dot{B}^{-2+\frac{d}{p}}_{p,1}}\}\leq c_{0},
\end{equation}
then the equations \eqref{NPP} admits a unique global solution. The initial condition \eqref{initial condition of npp} exhibits that we can obtain global existence of solutions for the equations \eqref{NPP} with
 only requiring the difference of initial particle densities of negative and positive electric charge is small enough, which reveals that the difference
of particle densities of negative and positive electric charge plays more important role in mathematical analysis of \eqref{NPP}. Since the space $\dot{B}^{-2+\frac{d}{2}}_{2,1}$ cannot contain the usual Sobolev space $\dot{H}^{-2+\frac{d}{2}}$, a natural question is that
\smallbreak

\textit{whether or not the global existence result in \cite{ZL23} holds in the Sobolev space $\dot{H}^{-2+\frac{d}{2}}$?}
\smallbreak

Back to the system \eqref{NSNPP}, to the best of our knowledge, mathematical analysis of system \eqref{NSNPP} was initially studied by Jerome \cite{J02},
where a local existence-uniqueness theory, based upon Kato's semigroup
framework for evolution equations,  was established. Notice that the right-hand side term of the momentum equations has a symmetric structure as
\begin{equation}\label{symmetric structure of phi}
 \Delta\phi\nabla\phi=\nabla\cdot\big(\nabla\phi\otimes\nabla\phi-\frac{1}{2}|\nabla\phi|^{2}I\big),
\end{equation}
one  can regard that the gradient of electrostatic potential plays the same role as the velocity field. Based on the observation \eqref{symmetric structure of phi}, by using the $L^p-L^q$ decay estimates of heat operator and the Hardy-Littlewood-Sobolev inequality, Zhao-Deng-Cui \cite{ZDC10, ZDC11} proved global well-posedness of \eqref{NSNPP}  with small initial
data in the critical Lebesgue spaces and Besov spaces.
\smallbreak

Notice that if we denote $v:=N-P$, $w:=N+P$, then we can transform \eqref{NSNPP} into the following equations with some symmetric structure:
\begin{equation}\label{tranformed NSNPP}
\begin{cases}
  \partial_{t} u+u\cdot\nabla u-\Delta
  u+\nabla \pi=-v\nabla(-\Delta)^{-1}v,\\
  \nabla\cdot u=0,\\
  \partial_{t} v+u\cdot \nabla
  v=\nabla\cdot(\nabla v+w\nabla(-\Delta)^{-1}v),\\
  \partial_{t} w+u\cdot \nabla
  w=\nabla\cdot(\nabla w+v\nabla(-\Delta)^{-1}v),\\
  u(x,0)=u_0(x), v(x,0)=v_0(x), w(x,0)=w_0(x),
\end{cases}
\end{equation}
where $v_{0}(x):=N_{0}(x)-P_{0}(x)$ and $w_{0}(x):=N_{0}(x)+P_{0}(x)$.  It is easy to find that the fourth equation of \eqref{tranformed NSNPP} is a linear equation for $w$ with coefficients depending on  $u$ and $v$, which may suggest us that we do not need to impose any smallness condition on initial data $w_{0}$ to ensure global existence of solutions. Indeed, Ma \cite{M18} proved that for some positive constants $1\leq  p<\infty$, $1\leq q<2d$, $q\leq 2p$ and
$\frac{1}{p}-\frac{1}{q}<\frac{1}{d}<\frac{1}{p}+\frac{1}{q}$, there exist two
positive constants $c_{0}$ and $C_{0}$ such that if
the initial data $(u_{0},v_{0}, w_{0})$ satisfies
\begin{equation}\label{large initial condition of Ma}
  \Big(\|u_{0}\|_{\dot{B}^{-1+\frac{d}{p}}_{p,1}}+\|v_{0}\|_{\dot{B}^{-2+\frac{d}{q}}_{q,1}}
  \Big)
  \exp\Big\{C_{0}\|w_{0}\|_{\dot{B}^{-2+\frac{d}{q}}_{q,1}}\Big\}\leq
  c_{0},
\end{equation}
then the equations \eqref{tranformed NSNPP} has a unique global solution. The condition  \eqref{large initial condition of Ma} implies global existence of solutions with only requiring  $u_{0}$ and $v_0$ are small enough (compared with $w_{0}$), which reveals that the velocity field  and  the
difference of electric charge densities play more important roles than the summation of electric charge densities in the mathematical analysis of \eqref{NSNPP}. Notice that $\dot{B}^{s}_{p,1}$ cannot contain the usual Sobolev space $\dot{H}^{s}$, similar question appeared that
\smallbreak

\textit{whether or not the global existence result in \cite{M18} holds in the Sobolev space $\dot{H}^{-1+\frac{d}{2}}\times (\dot{H}^{-2+\frac{d}{2}})^2$?}
\smallbreak

In this paper, we shall give the positive answers to the above two questions. Actually, we shall prove the global well-posedness of NS-NPP equations \eqref{NSNPP} in the critical homogeneous Besov spaces $\dot{B}^{-1+\frac{d}{p}}_{p,2}\times (\dot{B}^{-2+\frac{d}{q}}_{q,2})^2$, and the well-posed result in Sobolev spaces can be regarded as a special case in our main result. Before giving our main result, let us recall the definitions of the homogeneous Besov spaces and the Chemin-Lerner spaces (cf. \cite{BCD11, CL95}).  We denote the Schwartz class of rapidly
decreasing function by $\mathcal{S}$ and the space of tempered  distributions by $\mathcal{S}'$.

\begin{definition}\label{homo besov space}
 Let $\varphi\in\mathcal{S}$ be such that $\widehat{\varphi}(\xi)=1$ for $|\xi|\leq 1$ and $\widehat{\varphi}(\xi)=0$ for $|\xi|>2$. Denote, for
 $j\in \mathbb{Z}$, the function $\varphi_j(x):=2^{dj}\varphi(2^jx)$, we define the Littlewood-Paley operators $S_j f=\varphi_j\ast f$ and $\Delta_j f=S_{j+1}f-S_j f$. Let $f\in\mathcal{S}'$. Then $f$ belongs to the homogeneous Besov spaces $\dot{B}^{s}_{p,r}$  if and only if\
 \begin{itemize}
  \item The partial sum $\sum_{j=-m}^{j=m}\Delta_j f$ converges towards $f$ as a tempered distribution;

 \item The sequence $\{2^{js}\|\Delta_{j}f\|_{L^{p}}\}\in \ell^r$.
 \end{itemize}
 In that case, we define the norm as
\begin{equation*}
  \|f\|_{\dot{B}^{s}_{p,r}}:= \begin{cases} \big(\sum_{j\in\mathbb{Z}}2^{srj}\|\Delta_{j}f\|_{L^{p}}^{r}\big)^{\frac{1}{r}}
  \ \ &\text{for}\ \ 1\leq r<\infty,\\
  \sup_{j\in\mathbb{Z}}2^{sj}\|\Delta_{j}f\|_{L^{p}}\ \
  &\text{for}\ \
  r=\infty.
 \end{cases}
\end{equation*}
\end{definition}

\begin{definition}\label{chemin-lerner space}
 For $0<T\leq\infty$, $s\in \mathbb{R}$ and
$1\leq p, r, \rho\leq\infty$. We define the Chemin-Lerner space $ \widetilde{L}^{\rho}(0,T; \dot{B}^{s}_{p,r})$
as the completion of $\mathcal{C}([0,T], \mathcal{S})$ by the norm
$$
  \|f\|_{ \widetilde{L}^{\rho}_{T}(\dot{B}^{s}_{p,r})}:=\Big(\sum_{j\in\mathbb{Z}}2^{srj}\big(\int_{0}^{T}
  \|\Delta_{j}f(\cdot,t)\|_{L^{p}}^{\rho}dt\big)^{\frac{r}{\rho}}\Big)^{\frac{1}{r}}
$$
 with the standard modification for $\rho=\infty$ or $r=\infty$.
\end{definition}

Now we state our main result as  follows.

\begin{theorem}\label{main result in of NSNPP}
Let $d\geq 3$, $p$ and $q$ be two positive numbers such that $1\leq p, q<d$, $q\leq 2p$ and
$\frac{1}{p}-\frac{1}{q}< \frac{1}{d}$.
 There exist two constants $c_{0}$ and $C_{0}$ such that if the initial data $u_{0}\in
\dot{B}^{-1+\frac{d}{p}}_{p,2}$ with $\nabla\cdot
u_{0}=0$,  $N_{0}, P_{0}\in
\dot{B}^{-2+\frac{d}{q}}_{q,2}$, and  satisfy the condition:
\begin{equation}\label{initial condition in Besov space}
  \Big(\|u_{0}\|_{\dot{B}^{-1+\frac{d}{p}}_{p,2}}+\|N_{0}-P_{0}\|_{\dot{B}^{-2+\frac{d}{q}}_{q,2}}\Big)\exp\Big\{C_{0}\big(\|N_{0}+P_{0}\|_{\dot{B}^{-2+\frac{d}{q}}_{q,2}}^{2}+1\big)\Big\}\leq c_{0},
\end{equation}
then the NS-NPP equations \eqref{NSNPP} admits a unique global
solution $(u, N, P)$  satisfying
\begin{align*}
  u\in  \widetilde{L}^{\infty}(0, \infty;
  \dot{B}^{-1+\frac{d}{p}}_{p,2})\cap \widetilde{L}^{2}(0, \infty;
  \dot{B}^{\frac{d}{p}}_{p,2}),\ \ \ \
  N,P\in \widetilde{L}^{\infty}(0, \infty;
  \dot{B}^{-2+\frac{d}{q}}_{q,2})\cap \widetilde{L}^{2}(0, \infty;
  \dot{B}^{-1+\frac{d}{q}}_{q,2}).
\end{align*}
\end{theorem}

The main novelty of the present paper is that, to the best of our knowledge, we employ for the first time an $L^2$-type time weighted norm, in contrast to previous related studies which mostly rely on the classical $L^1$-type time weighted norm, to overcome the difficulties arising from the nonlinear terms involving $w$. To be specific, we define this new time weighted $L^2$ norm as follows: for  $s\in\mathbb{R}$, $p\in[1,\infty]$,  $f(t)\in L^{2}_{\operatorname{loc}}(0,+\infty)$, $f(t)\ge0$, define
\begin{equation*}
  \|u\|_{\widetilde{L}^{2}_{t,f}(\dot{B}^{s}_{p,2})}:=\Big(\sum_{j\in\mathbb{Z}}2^{2sj}
 \int_{0}^{t}f^{2}(\tau)\|\Delta_{j}u(\tau)\|_{L^{p}}^{2}d\tau\Big)^{\frac{1}{2}}.
\end{equation*}
 This modification enables us to provide an affirmative answer to the aforementioned questions regarding the global existence of solutions in Sobolev spaces.
 Moreover,  since the fourth equation of \eqref{tranformed NSNPP}  is a linear equation for $w$, we set
$$
 u_{\lambda,f}(x,t):=u(x,t)\exp\Big\{-\lambda\int_{0}^{t}f^{2}(\tau)d\tau\Big\}
$$
with the weighted function $ f(t):=\|w(\cdot,t)\|_{\dot{B}^{-1+\frac{d}{q}}_{q,2}}$ to eliminate the difficulties caused by the nonlinear term $\nabla\cdot(w\nabla (-\Delta)^{-1}v)$.
Such weighted function is totally different from that one in \cite{ZZL15, ZL23, ZL24}, and we can derive the desired weighted linear and bilinear estimates by localizing the weighted equations for $u_{\lambda,f}$ and $v_{\lambda,f}$ and using the Fubini's theorem to yield global existence of solutions under the initial condition \eqref{initial condition in Besov space} by the contradiction argument.
\smallbreak

As a corollary, taking $p=q=2$ in Theorem \ref{main result in of NSNPP} and using the fact $\dot{B}^{s}_{2,2}=\dot{H}^s$,  we obtain the following global well-posedness of the NS-NPP equations \eqref{NSNPP} in the framework of Sobolev spaces.

\begin{corollary}\label{Sobolev result of NSNPP}
There exist two constants $c_{0}$ and $C_{0}$ such that if the initial data $u_{0}\in
\dot{H}^{-1+\frac{d}{2}}$ with $\nabla\cdot
u_{0}=0$, $N_{0}, P_{0}\in
\dot{H}^{-2+\frac{d}{2}}$, and  satisfy the condition
\begin{equation}\label{initial condition in Sobolev space for NSNPP}
   \left(\|u_0\|_{\dot{H}^{-1+\frac{d}{2}}}+\|N_{0}-P_{0}\|_{\dot{H}^{-2+\frac{d}{2}}}\right)\exp\left\{C_{0}\big(\|N_{0}+P_{0}\|_{\dot{H}^{-2+\frac{d}{2}}}^2+1\big)\right\}
  \leq
  c_{0},
\end{equation}
then the NS-NPP equations \eqref{NSNPP} admits a unique global
solution $(u, N, P)$  satisfying
\begin{align*}
  u\in\widetilde{L}^{\infty}(0, \infty;
  \dot{H}^{-1+\frac{d}{2}})\cap \widetilde{L}^{2}(0, \infty;
  \dot{H}^{\frac{d}{2}}),\ \ \ \
  N,P\in  \widetilde{L}^{\infty}(0, \infty;
  \dot{H}^{-2+\frac{d}{2}})\cap \widetilde{L}^{2}(0, \infty;
 \dot{H}^{-1+\frac{d}{2}}).
\end{align*}
\end{corollary}

If we neglect the fluid motion in \eqref{NSNPP}, we obtain the following global well-posedness of the NPP equations \eqref{NPP} in the framework of Sobolev spaces.
\begin{corollary}\label{Sobolev result of NPP}
There exist two constants $c_{0}$ and $C_{0}$ such that if the initial data  $N_{0}, P_{0}\in
\dot{H}^{-2+\frac{d}{2}}$, and  satisfy the condition
\begin{equation}\label{initial condition in Sobolev space for NPP}
   \|N_{0}-P_{0}\|_{\dot{H}^{-2+\frac{d}{2}}}\exp\left\{C_{0}\big(\|N_{0}+P_{0}\|_{\dot{H}^{-2+\frac{d}{2}}}^2+1\big)\right\}
  \leq
  c_{0},
\end{equation}
then the NPP equations  \eqref{NPP} admits a unique global
solution $(N, P)$  satisfying
\begin{align*}
  N,P \in \widetilde{L}^{\infty}(0, \infty;
  \dot{H}^{-2+\frac{d}{2}})\cap \widetilde{L}^{2}(0, \infty;
 \dot{H}^{-1+\frac{d}{2}}).
\end{align*}
\end{corollary}

This paper is organized as follows. We shall present the proof of Theorem \ref{main result in of NSNPP} in the next section and sketch the proofs of  the local well-posedness and global well-posedness with small initial data in Appendix.

\section{Proof of Theorem \ref{main result in of NSNPP}}

In this section, we present the proof of Theorem \ref{main result in of NSNPP}. It is equivalent to prove that there exist two constants $c_{0}, C_{0}$ such that if
\begin{equation}\label{initial condition in Besov space1}
  \Big(\|u_{0}\|_{\dot{B}^{-1+\frac{d}{p}}_{p,2}}+\|v_{0}\|_{\dot{B}^{-2+\frac{d}{q}}_{q,2}}\Big)\exp\Big\{C_{0}\big(\|w_{0}\|_{\dot{B}^{-2+\frac{d}{q}}_{q,2}}^{2}+1\big)\Big\}\leq c_{0},
\end{equation}
then the system \eqref{tranformed NSNPP} admits a unique global
solution $(u, v, w)$ satisfying
\begin{align*}
  u\in  \widetilde{L}^{\infty}(0, \infty;
  \dot{B}^{-1+\frac{d}{p}}_{p,2})\cap \widetilde{L}^{2}(0, \infty;
  \dot{B}^{\frac{d}{p}}_{p,2}),\ \ \ \
  v,w\in \widetilde{L}^{\infty}(0, \infty;
  \dot{B}^{-2+\frac{d}{q}}_{q,2})\cap \widetilde{L}^{2}(0, \infty;
  \dot{B}^{-1+\frac{d}{q}}_{q,2}).
\end{align*}
\smallbreak

We shall use the following solvability result for the Cauchy problem of the heat equation (cf. \cite{BCD11}):
\begin{equation}\label{heat equation}
\begin{cases}
  \partial_{t}u-\Delta u= g, \ \
  &x\in\mathbb{R}^{d}, \ t>0,\\
  u(x,0)=u_{0}(x), \ \ &x\in\mathbb{R}^{d}.
\end{cases}
\end{equation}
\begin{proposition}\label{linear estimates of heat equation}
Let $s\in \mathbb{R}$, $1\leq p,r,\rho_1\leq\infty$ and
$0<T\leq\infty$. Assume that $u_{0}\in
\dot{B}^{s}_{p,r}$ and $ g \in\widetilde{L}^{\rho_1}_{T}(\dot{B}^{s+\frac{2}{\rho_{1}}-2}_{p,r})$. Then the heat equation \eqref{heat equation} admits
a unique solution
$$u\in\underset{\rho_1\leq
\rho\leq\infty}{\cap}\widetilde{L}^{\rho}_{T}(\dot{B}^{s+\frac{2}{\rho}}_{p,r}).$$
 In addition, there exists a
constant $C>0$ such that for any $\rho_1\leq
\rho\leq\infty$, we have
\begin{equation}\label{linear estimates of heat equation0}
  \|u\|_{\widetilde{L}^{\rho}_{T}(\dot{B}^{s+\frac{2}{\rho}}_{p,r})}\leq
  C\Big(\|u_{0}\|_{\dot{B}^{s}_{p,r}}+ \| g\|_{\widetilde{L}^{\rho_1}_{T}(\dot{B}^{s+\frac{2}{\rho_{1}}-2}_{p,r})}\Big).
\end{equation}
\end{proposition}

By Theorem \ref{well-posedness in Besov spaces} in Appendix, we know that there exists $T>0$ such that the system \eqref{tranformed NSNPP}  admits a unique
solution $(u,v,w)$ such that
\begin{align*}
  u\in  \widetilde{L}^{\infty}(0, T;
  \dot{B}^{-1+\frac{d}{p}}_{p,2})\cap \widetilde{L}^{2}(0, T;
  \dot{B}^{\frac{d}{p}}_{p,2}),\ \ \ \
  v,w\in \widetilde{L}^{\infty}(0, T;
  \dot{B}^{-2+\frac{d}{q}}_{q,2})\cap \widetilde{L}^{2}(0, T;
  \dot{B}^{-1+\frac{d}{q}}_{q,2}).
\end{align*}  Let us denote by $T_{*}$ the maximal existence time of this local solution. Then to prove Theorem \ref{main result in of NSNPP},
it suffices to prove $T_{*}=\infty$ under the initial condition \eqref{initial condition in Besov space1}.
By using the method of contradiction, we assume that $T_{*}<\infty$  and denote by $\eta$  a small enough positive constant which the exact value will be determined later, then we define $T_{\eta}$ as
\begin{align}\label{T_eta}
  T_{\eta}:=\max\Big\{t\in[0,T_{*}):\ &\|u\|_{\widetilde{L}^{\infty}_{t}(\dot{B}^{-1+\frac{d}{p}}_{p,2})}
  +\|u\|_{\widetilde{L}^{2}_{t}(\dot{B}^{\frac{d}{p}}_{p,2})}+\|v\|_{\widetilde{L}^{\infty}_{t}(\dot{B}^{-2+\frac{d}{q}}_{q,2})}
  +\|v\|_{\widetilde{L}^{2}_{t}(\dot{B}^{-1+\frac{d}{q}}_{q,2})}\leq
  \eta \Big\}.
\end{align}
We split the proof into the following three steps.
\smallbreak

\textit{Step 1. Estimate of $w$}.   Applying Proposition \ref{linear estimates of heat equation}, Lemma \ref{convective term of w} and Lemma \ref{drift term of w}, one can easily see that
\begin{align}\label{estimate of w1}
 &\|w\|_{\widetilde{L}^{\infty}_{t}(\dot{B}^{-2+\frac{d}{q}}_{q,2})}\!+\!\|w\|_{\widetilde{L}^{2}_{t}(\dot{B}^{-1+\frac{d}{q}}_{q,2})}\nonumber\\
   &\leq C\Big(\|w_{0}\|_{\dot{B}^{-2+\frac{d}{q}}_{q,2}}+\|u\cdot\nabla w+\nabla\cdot\left(v\nabla(-\Delta)^{-1}v\right)\|_{\widetilde{L}^{2}_{t}(\dot{B}^{-3+\frac{d}{q}}_{q,2})}\Big)\nonumber\\
   &\leq C\Big(\|w_{0}\|_{\dot{B}^{-2+\frac{d}{q}}_{q,2}}+\|u\cdot \nabla w\|_{\widetilde{L}^{2}_{t}(\dot{B}^{-3+\frac{d}{q}}_{q,2})}+\|v\nabla(-\Delta)^{-1}v\|_{\widetilde{L}^{2}_{t}(\dot{B}^{-2+\frac{d}{q}}_{q,2})}\Big)\nonumber\\
   &\leq C\Big(\|w_{0}\|_{\dot{B}^{-2+\frac{d}{q}}_{q,2}}+\|u\|_{\widetilde{L}^{2}_{t}(\dot{B}^{\frac{d}{p}}_{p,2})}\|w\|_{\widetilde{L}^{\infty}_{t}(\dot{B}^{-2+\frac{d}{q}}_{q,2})}
   +\|u\|_{\widetilde{L}^{\infty}_{t}(\dot{B}^{-1+\frac{d}{p}}_{p,2})}
   \|w\|_{\widetilde{L}^{2}_{t}(\dot{B}^{-1+\frac{d}{q}}_{q,2})}\nonumber\\
   &\quad\quad\ \ +\|v\|_{\widetilde{L}^{\infty}_{t}(\dot{B}^{-2+\frac{d}{q}}_{q,2})}
   \|v\|_{\widetilde{L}^{2}_{t}(\dot{B}^{-1+\frac{d}{q}}_{q,2})}\Big).
\end{align}
Choosing  $\eta$ small enough such that $2C\eta<1$ and using \eqref{T_eta}, we get
\begin{align}\label{estimate of w2}
 \|w\|_{\widetilde{L}^{\infty}_{t}(\dot{B}^{-2+\frac{d}{q}}_{q,2})}+\|w\|_{\widetilde{L}^{2}_{t}(\dot{B}^{-1+\frac{d}{q}}_{q,2})}
 \leq 2C\|w_{0}\|_{\dot{B}^{-2+\frac{d}{q}}_{q,2}}+\eta.
\end{align}
\smallbreak

\textit{Step 2. Estimate of $v$}.   Let $\lambda$ be a positive constant which the exact value will be specified later. Then using the Duhamel's principle and taking the weight function $f(t)$ into consideration, we can see that $v_{\lambda,f}$ satisfies the following integral equation:
\begin{equation}\label{integral equation of v-f}
v_{\lambda,f}=e^{-\lambda\int_{0}^{t}f^{2}(\tau)d\tau}e^{t\Delta}v_{0}+\int_{0}^{t}e^{-\lambda\int_{s}^{t}f(\tau)d\tau}e^{(t-s)\Delta}\left(u\cdot\nabla v_{\lambda,f}+\nabla\cdot(w\nabla(-\Delta)^{-1}v_{\lambda,f})\right)ds.
\end{equation}
Notice that $e^{-\lambda\int_{s}^{t}f^{2}(\tau)d\tau}\leq 1$ for any $0\leq s\leq t$. Thus we can deduce from  Proposition \ref{linear estimates of heat equation} that
\begin{align}\label{estimate of v-f0}
   \|v_{\lambda,f}\|_{\widetilde{L}^{\infty}_{t}(\dot{B}^{-2+\frac{d}{q}}_{q,2})}
   &+\|v_{\lambda,f}\|_{\widetilde{L}^{2}_{t}(\dot{B}^{-1+\frac{d}{q}}_{q,2})}\nonumber\\
&\leq C\Big(\|v_{0}\|_{\dot{B}^{-2+\frac{d}{q}}_{q,2}}+ \|u\cdot \nabla v_{\lambda,f}+\nabla\cdot\left(w\nabla(-\Delta)^{-1}v_{\lambda,f}\right)\|_{\widetilde{L}^{2}_{t}(\dot{B}^{-3+\frac{d}{q}}_{q,2})}\Big)\nonumber\\
   &\leq C\Big(\|v_{0}\|_{\dot{B}^{-2+\frac{d}{q}}_{q,2}}+\|u\cdot\nabla v_{\lambda,f}\|_{\widetilde{L}^{2}_{t}(\dot{B}^{-3+\frac{d}{q}}_{q,2})}+\|w\nabla(-\Delta)^{-1}v_{\lambda,f}\|_{\widetilde{L}^{2}_{t}(\dot{B}^{-2+\frac{d}{q}}_{q,2})}\Big).
\end{align}
On the other hand, applying the dyadic operator $\Delta_{j}$ to \eqref{integral equation of v-f} and taking
$L^{q}$ norm to the resulting equation, we see that there exists a positive constant $\kappa$ such that
\begin{align}\label{estimate of v-f1}
   \|\Delta_{j}v_{\lambda,f}\|_{L^{q}}&\leq e^{-\lambda\int_{0}^{t}f^{2}(\tau)d\tau-\kappa t2^{2j}} \|\Delta_{j}v_{0}\|_{L^{q}}\nonumber\\
   &+\int_{0}^{t}e^{-\lambda\int_{s}^{t}f^{2}(\tau)d\tau-\kappa (t-s)2^{2j}}\|\Delta_{j}\left(u\cdot\nabla v_{\lambda,f}+\nabla\cdot(w\nabla(-\Delta)^{-1}v_{\lambda,f})\right)\|_{L^{q}}ds.
\end{align}
Multiplying \eqref{estimate of v-f1} by $\sqrt{\lambda} f(t)$ and taking $L^{2}$ norm with respect to time variable $t$ to the resulting inequality, we see that
\begin{align}\label{estimate of v-f2}
 &\sqrt{\lambda}\big\|f(\cdot)\|\Delta_{j}v_{\lambda,f}\|_{L^{q}}\big\|_{L^{2}_{t}}\leq \frac{1}{\sqrt{2}}\|\Delta_{j}v_{0}\|_{L^{q}}\nonumber\\
 &+\sqrt{\lambda} \Big\|f(t)\int_{0}^{t}e^{-\lambda\int_{s}^{t}f^{2}(\tau)d\tau-\kappa (t-s)2^{2j}}\|\Delta_{j}\big(u\cdot\nabla v_{\lambda,f}+\nabla\cdot(w\nabla(-\Delta)^{-1}v_{\lambda,f})\big)\|_{L^{q}}ds\Big\|_{L^{2}_{t}},
\end{align}
where we have used the fact that
\begin{align*}
\Big\|\sqrt{\lambda} f(t) e^{-\lambda\int_{0}^{t}f^{2}(\tau)d\tau-\kappa t2^{2j}}\Big\|_{L^{2}_{t}}&\leq \Big(\int_{0}^{t}e^{-2\lambda\int_{0}^{s}f^{2}(\tau)d\tau}\lambda f^{2}(s)ds\Big)^{\frac{1}{2}}\nonumber\\
&\leq \Big(\int_{0}^{t}e^{-2\lambda\int_{0}^{s}f^{2}(\tau)d\tau}d\big(\lambda\int_{0}^{s}f^{2}(\tau)d\tau\big)\Big)^{\frac{1}{2}}\nonumber\\
&\leq \frac{1}{\sqrt{2}}.
\end{align*}
To bound the second term on the right-hand side of \eqref{estimate of v-f2}, we first apply the H\"{o}lder's inequality to get
\begin{align}\label{estimate of v-f3}
&\sqrt{\lambda}f(t)\int_{0}^{t}e^{-\lambda\int_{s}^{t}f^{2}(\tau)d\tau-\kappa (t-s)2^{2j}}\|\Delta_{j}\big(u\cdot\nabla v_{\lambda,f}+\nabla\cdot(w\nabla(-\Delta)^{-1}v_{\lambda,f})\big)\|_{L^{q}}ds\nonumber\\
&\leq \Big(\int_{0}^{t}\lambda f^{2}(t) e^{-2\lambda\int_{s}^{t}f^{2}(\tau)d\tau}\|\Delta_{j}\big(u\cdot\nabla v_{\lambda,f}+\nabla\cdot(w\nabla(-\Delta)^{-1}v_{\lambda,f})\big)\|_{L^{q}}^{2}ds\Big)^{\frac{1}{2}}\Big(\int_{0}^{t}e^{-2\kappa (t-s)2^{2j}}ds\Big)^{\frac{1}{2}}\nonumber\\
&\leq C2^{-j}\Big(\int_{0}^{t}\lambda f^{2}(t) e^{-2\lambda\int_{s}^{t}f^{2}(\tau)d\tau}\|\Delta_{j}\big(u\cdot\nabla v_{\lambda,f}+\nabla\cdot(w\nabla(-\Delta)^{-1}v_{\lambda,f})\big)\|_{L^{q}}^{2}ds\Big)^{\frac{1}{2}}.
\end{align}
Secondly, by taking $L^{2}$ norm with respect to time variable $t$ on $[0,T]$ (where $T$ is just a sign to distinct $t$), one gets
\begin{align}\label{estimate of v-f4}
&\Big\|\sqrt{\lambda}f(t)\int_{0}^{t}e^{-\lambda\int_{s}^{t}f^{2}(\tau)d\tau-\kappa (t-s)2^{2j}}\|\Delta_{j}\big(u\cdot\nabla v_{\lambda,f}+\nabla\cdot(w\nabla(-\Delta)^{-1}v_{\lambda,f})\big)\|_{L^{q}}ds\Big\|_{L^{2}_{t}}\nonumber\\
&\leq C2^{-j}\Big(\int_{0}^{T}\int_{0}^{t}\lambda f^{2}(t) e^{-2\lambda\int_{s}^{t}f^{2}(\tau)d\tau}\|\Delta_{j}\big(u\cdot\nabla v_{\lambda,f}+\nabla\cdot(w\nabla(-\Delta)^{-1}v_{\lambda,f})\big)\|_{L^{q}}^{2}dsdt\Big)^{\frac{1}{2}}.
\end{align}
Using Fubini's theorem, we obtain from \eqref{estimate of v-f4} that
\begin{align}\label{estimate of v-f5}
&\Big\|\sqrt{\lambda}f(t)\int_{0}^{t}e^{-\lambda\int_{s}^{t}f^{2}(\tau)d\tau-\kappa (t-s)2^{2j}}\|\Delta_{j}\big(u\cdot\nabla v_{\lambda,f}+\nabla\cdot(w\nabla(-\Delta)^{-1}v_{\lambda,f})\big)\|_{L^{q}}ds\Big\|_{L^{2}_{t}}\nonumber\\
&\leq C2^{-j}\Big(\int_{0}^{T}\lambda f^{2}(t)e^{-2\lambda\int_{s}^{t}f^{2}(\tau)d\tau}dt\Big)^{\frac{1}{2}}
\|\Delta_{j}\big(u\cdot\nabla v_{\lambda,f}+\nabla\cdot(w\nabla(-\Delta)^{-1}v_{\lambda,f})\big)\|_{L^{2}_{t}(L^{q})}\nonumber\\
&\leq C2^{-j}\|\Delta_{j}\big(u\cdot\nabla v_{\lambda,f}+\nabla\cdot(w\nabla(-\Delta)^{-1}v_{\lambda,f})\big)\|_{L^{2}_{t}(L^{q})}.
\end{align}
Taking \eqref{estimate of v-f5} into \eqref{estimate of v-f2}, we obtain
\begin{align}\label{estimate of v-f6}
 &\sqrt{\lambda}\Big\|f(\cdot)\big\|\Delta_{j}v_{\lambda,f}\big\|_{L^{q}}\Big\|_{L^{2}_{t}}\leq \frac{1}{\sqrt{2}}\|\Delta_{j}v_{0}\|_{L^{p}}+C2^{-j}\Big\|\Delta_{j}\big(u\cdot\nabla v_{\lambda,f}+\nabla\cdot(w\nabla(-\Delta)^{-1}v_{\lambda,f})\big)\Big\|_{L^{2}_{t}(L^{q})}.
\end{align}
Multiplying $2^{(-2+\frac{d}{q})j}$ to \eqref{estimate of v-f6}, then taking $\ell^{2}$ norm with respect to $j$, one sees that
\begin{align}\label{estimate of v-f7}
 &\sqrt{\lambda}\|v_{\lambda,f}\|_{\widetilde{L}^{2}_{t,f}(\dot{B}^{-2+\frac{d}{q}}_{q,2})}\leq \|v_{0}\|_{\dot{B}^{-2+\frac{d}{q}}_{q,2}}+C \|u\cdot\nabla v_{\lambda,f}+\nabla\cdot(w\nabla(-\Delta)^{-1}v_{\lambda,f})\|_{\widetilde{L}^{2}_{t}(\dot{B}^{-3+\frac{d}{q}}_{q,2})}.
\end{align}
Now putting \eqref{estimate of v-f0} and \eqref{estimate of v-f7} together, we obtain
\begin{align}\label{estimate of v-f}
   \|v_{\lambda,f}\|_{\widetilde{L}^{\infty}_{t}(\dot{B}^{-2+\frac{d}{q}}_{q,2})}&
   +\sqrt{\lambda}\|v_{\lambda,f}\|_{\widetilde{L}^{2}_{t,f}(\dot{B}^{-2+\frac{d}{q}}_{q,2})}
   +\|v_{\lambda,f}\|_{\widetilde{L}^{2}_{t}(\dot{B}^{-1+\frac{d}{q}}_{q,2})}\nonumber\\
   &\leq C\Big(\|v_{0}\|_{\dot{B}^{-2+\frac{d}{q}}_{q,2}}+\|u\cdot\nabla v_{\lambda,f}\|_{\widetilde{L}^{2}_{t}(\dot{B}^{-3+\frac{d}{q}}_{q,2})}
   +\|w\nabla(-\Delta)^{-1}v_{\lambda,f}\|_{\widetilde{L}^{2}_{t}(\dot{B}^{-2+\frac{d}{q}}_{q,2})}\Big).
\end{align}
By Lemma \ref{convective term of w}, we can bound the second term on the right-hand side of \eqref{estimate of v-f} that
\begin{align}\label{estimate of uv-f}
   \|u\cdot \nabla v_{\lambda,f}\|_{\widetilde{L}^{2}_{t}(\dot{B}^{-3+\frac{d}{q}}_{q,2})}\leq C \Big(\|u\|_{\widetilde{L}^{2}_{t}(\dot{B}^{\frac{d}{p}}_{p,2})} \|v_{\lambda,f}\|_{\widetilde{L}^{\infty}_{t}(\dot{B}^{-2+\frac{d}{q}}_{q,2})}+ \|u\|_{\widetilde{L}^{\infty}_{t}(\dot{B}^{-1+\frac{d}{p}}_{p,2})} \|v_{\lambda,f}\|_{\widetilde{L}^{2}_{t}(\dot{B}^{-1+\frac{d}{q}}_{q,2})}\Big).
\end{align}
Using the Minkowski's inequality, we can use Lemma \ref{product estimates in Besov spaces} to estimate the third term on the right-hand side of \eqref{estimate of v-f} that
\begin{align}\label{estimate of wv-f}
   \|w\nabla(-\Delta)^{-1}v_{\lambda,f}\|_{\widetilde{L}^{2}_{t}(\dot{B}^{-2+\frac{d}{q}}_{q,2})}&\approx \|w\nabla(-\Delta)^{-1}v_{\lambda,f}\|_{L^{2}_{t}(\dot{B}^{-2+\frac{d}{q}}_{q,2})}\nonumber\\&
   =\Big(\int_{0}^{t}\|w\nabla(-\Delta)^{-1}v_{\lambda,f}\|_{\dot{B}^{-2+\frac{d}{q}}_{q,2}}^{2}ds\Big)^{\frac{1}{2}}\nonumber\\
   &\leq C\Big(\int_{0}^{t}\|w\|_{\dot{B}^{-1+\frac{d}{q}}_{q,2}}^{2}\|v_{\lambda,f}\|_{\dot{B}^{-2+\frac{d}{q}}_{q,2}}^{2}ds\Big)^{\frac{1}{2}}\nonumber\\
   &\leq C \|v_{\lambda,f}\|_{\widetilde{L}^{2}_{t,f}(\dot{B}^{-2+\frac{d}{q}}_{q,2})}.
\end{align}
Taking \eqref{estimate of uv-f} and \eqref{estimate of wv-f} into  \eqref{estimate of v-f} , we know that
\begin{align}\label{estimate of v-f10}
   &\|v_{\lambda,f}\|_{\widetilde{L}^{\infty}_{t}(\dot{B}^{-2+\frac{d}{q}}_{q,2})}
   +\sqrt{\lambda}\|v_{\lambda,f}\|_{\widetilde{L}^{2}_{t,f}(\dot{B}^{-2+\frac{d}{q}}_{q,2})}
     +\|v_{\lambda,f}\|_{\widetilde{L}^{2}_{t}(\dot{B}^{-1+\frac{d}{q}}_{q,2})}\leq C\Big(\|v_{0}\|_{\dot{B}^{-2+\frac{d}{q}}_{q,2}} \nonumber\\
     &\quad\quad+\|u\|_{\widetilde{L}^{2}_{t}(\dot{B}^{\frac{d}{p}}_{p,2})} \|v_{\lambda,f}\|_{\widetilde{L}^{\infty}_{t}(\dot{B}^{-2+\frac{d}{q}}_{q,2})}+ \|u\|_{\widetilde{L}^{\infty}_{t}(\dot{B}^{-1+\frac{d}{p}}_{p,2})} \|v_{\lambda,f}\|_{\widetilde{L}^{2}_{t}(\dot{B}^{-1+\frac{d}{q}}_{q,2})}+\|v_{\lambda,f}\|_{\widetilde{L}^{2}_{t,f}(\dot{B}^{-2+\frac{d}{q}}_{q,2})}\Big).
\end{align}
Therefore, by choosing $\lambda$ large enough such that $\sqrt{\lambda}>C$ and $\eta$ small enough such that $2C\eta<1$,  we obtain
\begin{align}\label{estimate of v-f11}
   \|v_{\lambda,f}\|_{\widetilde{L}^{\infty}_{t}(\dot{B}^{-2+\frac{d}{q}}_{q,2})}
   +\|v_{\lambda,f}\|_{\widetilde{L}^{2}_{t}(\dot{B}^{-1+\frac{d}{q}}_{q,2})}
   \leq 2C\|v_{0}\|_{\dot{B}^{-2+\frac{d}{q}}_{q,2}}.
\end{align}

\textit{Step 3. Estimate of $u$}.  Notice that the weighted function $u_{\lambda,f}$ satisfies the following equivalent integral equation:
\begin{equation}\label{weighted equation of u}
u_{\lambda,f}=e^{-\lambda\int_{0}^{t}f^{2}(\tau)d\tau}e^{t\Delta}u_{0}-\int_{0}^{t}e^{-\lambda\int_{s}^{t}f(\tau)d\tau}e^{(t-s)\Delta}\mathbb{P}(u\cdot\nabla u_{\lambda,f}-v\nabla(-\Delta)^{-1}v_{\lambda,f})ds,
\end{equation}
where $\mathbb{P}=I+\nabla(-\Delta)^{-1}\operatorname{div}$ is the Leray projector,
i.e., the $d\times d$ matrix pseudo-differential operator in
$\mathbb{R}^d$ with the symbol
$(\delta_{ij}-\frac{\xi_i\xi_j}{|\xi|^2})_{i,j=1}^d$. Since  $\mathbb{P}$ is a bounded linear operator in the homogeneous Besov spaces, we obtain from Proposition \ref{linear estimates of heat equation} that
\begin{align}\label{estimate of weighted u}
   \|u_{\lambda,f}\|_{\widetilde{L}^{\infty}_{t}(\dot{B}^{-1+\frac{d}{p}}_{p,2})}
   +\|u_{\lambda,f}\|_{\widetilde{L}^{2}_{t}(\dot{B}^{\frac{d}{p}}_{p,2})}
   \leq C\Big(\|u_{0}\|_{\dot{B}^{-1+\frac{d}{p}}_{p,2}}+\|u\cdot\nabla u_{\lambda,f}\!-\!v\nabla(-\Delta)^{-1}v_{\lambda,f}\|_{\widetilde{L}^{2}_{t}(\dot{B}^{-2+\frac{d}{p}}_{p,2})}\Big).
\end{align}
For the first term on the right-hand side of \eqref{estimate of weighted u},  by using Lemma \ref{convective term in NS}, we can bound it as
\begin{align*}
   \|u\cdot\nabla u_{\lambda,f}\|_{\widetilde{L}^{2}_{t}(\dot{B}^{-2+\frac{d}{p}}_{p,2})}\leq C\|u\|_{\widetilde{L}^{2}_{t}(\dot{B}^{\frac{d}{p}}_{p,2})}\|u_{\lambda,f}\|_{\widetilde{L}^{\infty}_{t}(\dot{B}^{-1+\frac{d}{p}}_{p,2})};
\end{align*}
while for the second term on the right-hand side of \eqref{estimate of weighted u},  by using Lemma \ref{external force in NS}, we get
\begin{align*}
   \| v\nabla(-\Delta)^{-1}v_{\lambda,f}\|_{\widetilde{L}^{2}_{t}(\dot{B}^{-2+\frac{d}{p}}_{p,2})}\leq C\|v\|_{\widetilde{L}^{2}_{t}(\dot{B}^{-1+\frac{d}{q}}_{q,2})}\|v_{\lambda,f}\|_{\widetilde{L}^{\infty}_{t}(\dot{B}^{-2+\frac{d}{q}}_{q,2})}.
\end{align*}
Taking these two estimates into \eqref{estimate of weighted u}, and choosing $\eta$ small enough such that $4C\eta<1$, we get
\begin{align}\label{estimate of weighted u1}
  & \|u_{\lambda,f}\|_{\widetilde{L}^{\infty}_{t}(\dot{B}^{-1+\frac{d}{p}}_{p,2})}
   +\|u_{\lambda,f}\|_{\widetilde{L}^{2}_{t}(\dot{B}^{\frac{d}{p}}_{p,2})}\nonumber\\
   &
   \leq C\Big(\|u_{0}\|_{\dot{B}^{-1+\frac{d}{p}}_{p,2}}+\|u\|_{\widetilde{L}^{2}_{t}(\dot{B}^{\frac{d}{p}}_{p,2})}\|u_{\lambda,f}\|_{\widetilde{L}^{\infty}_{t}(\dot{B}^{-1+\frac{d}{p}}_{p,2})}
   +\|v\|_{\widetilde{L}^{2}_{t}(\dot{B}^{-1+\frac{d}{q}}_{q,2})}\|v_{\lambda,f}\|_{\widetilde{L}^{\infty}_{t}(\dot{B}^{-2+\frac{d}{q}}_{q,2})}\Big)\nonumber\\
   &\leq C\|u_{0}\|_{\dot{B}^{-1+\frac{d}{p}}_{p,2}}+\frac{1}{4}\|u_{\lambda,f}\|_{\widetilde{L}^{\infty}_{t}(\dot{B}^{-1+\frac{d}{p}}_{p,2})}
   +\frac{1}{4}\|v_{\lambda,f}\|_{\widetilde{L}^{\infty}_{t}(\dot{B}^{-2+\frac{d}{q}}_{q,2})}.
\end{align}
\smallbreak

Finally, putting \eqref{estimate of v-f11} and \eqref{estimate of weighted u1} together, we obtain from \eqref{estimate of w2} that
\begin{align}\label{estimate of uv}
  \|u\|_{\widetilde{L}^{\infty}_{t}(\dot{B}^{-1+\frac{d}{p}}_{p,2})}
   & +\|u\|_{\widetilde{L}^{2}_{t}(\dot{B}^{\frac{d}{p}}_{p,2})}+ \|v\|_{\widetilde{L}^{\infty}_{t}(\dot{B}^{-2+\frac{d}{q}}_{q,2})}
   +\|v\|_{\widetilde{L}^{2}_{t}(\dot{B}^{-1+\frac{d}{q}}_{q,2})}\nonumber\\
   &\leq C\Big(\|u_{0}\|_{\dot{B}^{-1+\frac{d}{p}}_{p,2}}+\|v_{0}\|_{\dot{B}^{-2+\frac{d}{q}}_{q,2}}\Big)\exp\Big\{\lambda\int_{0}^{t}\|w(\tau)\|_{\dot{B}^{-1+\frac{d}{q}}_{q,2}}^{2}d\tau\Big\}\nonumber\\
   &\leq C\Big(\|u_{0}\|_{\dot{B}^{-1+\frac{d}{p}}_{p,2}}+\|v_{0}\|_{\dot{B}^{-2+\frac{d}{q}}_{q,2}}\Big)\exp\Big\{C
   (\|w_{0}\|_{\dot{B}^{-2+\frac{d}{q}}_{q,2}}^{2}+1)\Big\}.
\end{align}
Thus we conclude that if we take $C_{0}$ large enough and $c_{0}$ small enough in \eqref{initial condition in Besov space1}, then it follows from \eqref{estimate of uv} that
\begin{align*}
     \|u\|_{\widetilde{L}^{\infty}_{t}(\dot{B}^{-1+\frac{d}{p}}_{p,2})}
   +\|u\|_{\widetilde{L}^{2}_{t}(\dot{B}^{\frac{d}{p}}_{p,2})}+ \|v\|_{\widetilde{L}^{\infty}_{t}(\dot{B}^{-2+\frac{d}{q}}_{q,2})}
   +\|v\|_{\widetilde{L}^{2}_{t}(\dot{B}^{-1+\frac{d}{q}}_{q,2})}
   \leq \frac{\eta}{2}
\end{align*}
for all $t<T_{\eta}$, which contradicts with the maximality of $T_{\eta}$, thus $T^*=\infty$. We complete the proof of Theorem \ref{main result in of NSNPP}.

\section{Appendix}

In this Appendix, we first recall some well-known results in the  homogeneous Besov spaces and the Chemin-Lerner spaces, then we sketch the proof of local well-posedness and global well-posedness with small
initial data of the tranformed NS-NPP  system  \eqref{tranformed NSNPP}.

\subsection{Analytical tools in Besov spaces}

 Let us first recall the well-known Bernstein's inequalities and some  properties of the homogeneous Besov spaces. For more details, see \cite{BCD11, T83}.

\begin{lemma}\label{Bernstein inequality}
Let $\mathcal{B}$ be a ball, and $\mathcal{C}$  a ring in
$\mathbb{R}^{d}$. There exists a constant $C$ such that for any
positive real number $\lambda$, any nonnegative integer $k$ and any
couple of real numbers $(a,b)$ with $1\leq a\le b\leq \infty$, we
have
\begin{equation}\label{Bernstein inequality1}
   \operatorname{supp}\hat{f}\subset\lambda\mathcal{B}\ \ \Rightarrow\ \   \sup_{|\alpha|=k}\|\partial^{\alpha}f\|_{L^{b}}\leq
   C^{k+1}\lambda^{k+d(\frac{1}{a}-\frac{1}{b})}\|f\|_{L^{a}},
\end{equation}
\begin{equation}\label{Bernstein inequality2}
   \operatorname{supp}\hat{f}\subset\lambda\mathcal{C} \ \ \Rightarrow\ \   C^{-1-k}\lambda^{k}\|f\|_{L^{a}}\leq
   \sup_{|\alpha|=k}\|\partial^{\alpha}f\|_{L^{a}}\leq  C^{1+k}\lambda^{k}\|f\|_{L^{a}}.
\end{equation}
\end{lemma}

\begin{lemma}\label{Properties}
The following properties hold:
\begin{itemize}

\item [i)] Derivatives: There exists a universal constant $C$ such that
\begin{equation*}
    C^{-1}\|u\|_{\dot{B}^{s}_{p,r}}\leq \|\nabla u\|_{\dot{B}^{s-1}_{p,r}}\leq C\|u\|_{\dot{B}^{s}_{p,r}}.
\end{equation*}

\item [ii)] Fractional derivative: Let $\Lambda:=\sqrt{-\Delta}$ and $\sigma\in\mathbb{R}$. Then the operator $\Lambda^{\sigma}$ is an isomorphism from $\dot{B}^{s}_{p,r}$
to $\dot{B}^{s-\sigma}_{p,r}$.

\item [iii)] Imbedding: For $1\leq p_{1}\leq p_{2}\leq \infty$ and $1\leq r_{1}\leq r_{2}\leq \infty$, we have the continuous imbedding  $\dot{B}^{s}_{p_{1},r_{1}}\hookrightarrow \dot{B}^{s-d(\frac{1}{p_{1}}-\frac{1}{p_{2}})}_{p_{2},r_{2}}$.

\item [iv)] Interpolation:  For  $s_{1},s_{2}\in \mathbb{R}$ such that $s_{1}<s_{2}$,  $\theta\in [0,1]$ and $1\leq p, r\leq \infty$, there exists a constant $C$
such that
\begin{align}\label{interpolation in Besov spaces}
    \|u\|_{\dot{B}^{s_{1}\theta+s_{2}(1-\theta)}_{p,r}}\leq C\|u\|_{\dot{B}^{s_{1}}_{p,r}}^{\theta}\|u\|_{\dot{B}^{s_{2}}_{p,r}}^{1-\theta}.
\end{align}
\end{itemize}
\end{lemma}

It is easy to extend the above interpolation inequality \eqref{interpolation in Besov spaces} to the Chemin-Lerner spaces:  for any  $s_{1}<s_{2}$,  $\theta\in [0,1]$, and $1\leq p, r,\rho, \rho_1,\rho_2\leq \infty$ with $\frac{1}{\rho}=\frac{1}{\rho_1}+\frac{1}{\rho_2}$,  there exists a constant $C$
such that
\begin{align}\label{interpolation in Chemin-Lerner spaces}
    \|u\|_{\widetilde{L}^{\rho}(\dot{B}^{s_{1}\theta+s_{2}(1-\theta)}_{p,r})}\leq C\|u\|_{\widetilde{L}^{\rho_1}(\dot{B}^{s_{1}}_{p,r})}^{\theta}\|u\|_{\widetilde{L}^{\rho_2}(\dot{B}^{s_{2}}_{p,r})}^{1-\theta}.
\end{align}
\smallbreak

Next, we also need to recall the Bony's decomposition from \cite{B81}:
\begin{equation}\label{Bony decomposition}
  fg=T_{f}g+T_{g}f+R(f,g),
\end{equation}
where  $T_{f}g:=\sum_{j\in\mathbb{Z}}S_{j-1}f\Delta_{j}g$ is the paraproduct of $f$ and $g$, $R(f,g):=\sum_{j\in\mathcal{Z}}\Delta_{j}f\widetilde{\Delta}_{j}g$ is
the remainder of $f$ and $g$ with
\begin{equation*}
\widetilde{\Delta}_{j}g:=\sum_{|j-j'|\leq 1}\Delta_{j'}g.
\end{equation*}
\smallbreak

Finally, let us prove the following crucial estimates for the product of two functions in the homogeneous Besov spaces and the Chemin-Lerner spaces.

\begin{lemma}\label{product estimates in Besov spaces}  Let $1\leq p_{1}, p_{2}\leq\infty$, $s_{1}<\frac{d}{p_1}$,
$s_2<\min\{\frac{d}{p_{1}}, \frac{d}{p_2}\}$ with
$s_1+s_2>d\max(0,\frac{1}{p_1}+\frac{1}{p_2}-1)$. Assume that $f\in
\dot{B}^{s_1}_{p_1,2}$,
$g\in\dot{B}^{s_2}_{p_2,2}$. Then
$fg\in\dot{B}^{s_1+s_2-\frac{d}{p_1}}_{p_2,2}$, and
there exists a positive constant $C$ such that
\begin{equation}\label{product estimates in Besov spaces1}
   \|fg\|_{\dot{B}^{s_1+s_2-\frac{d}{p_1}}_{p_2,2}}\leq C\|f\|_{\dot{B}^{s_1}_{p_1,2}}\|g\|_{\dot{B}^{s_2}_{p_2,2}}.
\end{equation}
\end{lemma}
\begin{proof}
The proof is essentially coming from Lemma 5.3 in \cite{ZL17}, where the authors proved that, under the same assumptions on $p_1, p_2, s_1, s_2$, the following product estimate holds:
\begin{equation*}
   \|fg\|_{\dot{B}^{s_1+s_2-\frac{d}{p_1}}_{p_2,1}}\leq C\|f\|_{\dot{B}^{s_1}_{p_1,1}}\|g\|_{\dot{B}^{s_2}_{p_2,1}}.
\end{equation*}
This immediately yields \eqref{product estimates in Besov spaces1} by only using the H\"{o}lder's inequality and the fact $\ell^2\hookrightarrow\ell^4$.
\end{proof}

\begin{lemma}\label{product estimates in Chemin-Lerner spaces}  Let $1\leq p_{1}, p_{2}\leq\infty$, $s_{1}<\frac{d}{p_1}$,
$s_2<\min\{\frac{d}{p_{1}}, \frac{d}{p_2}\}$ with
$s_1+s_2>d\max(0,\frac{1}{p_1}+\frac{1}{p_2}-1)$, $1\leq \rho, \rho_1, \rho_2\leq \infty$ with $\frac{1}{\rho}=\frac{1}{\rho_1}+\frac{1}{\rho_2}$. Assume that $f\in
\widetilde{L}_T^{\rho_1}(\dot{B}^{s_1}_{p_1,2})$,
$g\in\widetilde{L}_T^{\rho_2}(\dot{B}^{s_2}_{p_2,2})$ for any $0<T\leq \infty$. Then
$fg\in\widetilde{L}_T^{\rho}(\dot{B}^{s_1+s_2-\frac{d}{p_1}}_{p_2,2})$, and
there exists a positive constant $C$ such that
\begin{equation}\label{product estimates in Chemin-Lerner spaces1}
   \|fg\|_{\widetilde{L}_T^{\rho}(\dot{B}^{s_1+s_2-\frac{d}{p_1}}_{p_2,2})}\leq C\|f\|_{\widetilde{L}_T^{\rho_1}(\dot{B}^{s_1}_{p_1,2})}\|g\|_{\widetilde{L}_T^{\rho_2}(\dot{B}^{s_2}_{p_2,2})}.
\end{equation}
\end{lemma}
\begin{proof}
This is a direct corollary of Lemma \ref{product estimates in Besov spaces}, where the time integrable indices $\rho$, $\rho_1$ and $\rho_2$ obey the rule of the H\"{o}lder's inequality.
\end{proof}

\subsection{Wellposedness of the tranformed NS-NPP system}

 In this subsection, we prove the following well-posedness result to the tranformed NS-NPP  system  \eqref{tranformed NSNPP}.

\begin{theorem}\label{well-posedness in Besov spaces} Let $d\geq 3$, $p$ and $q$ be two positive numbers such that $1\leq p, q<d$, $q\leq 2p$ and
$\frac{1}{p}-\frac{1}{q}< \frac{1}{d}$.
Assume that $u_{0}\in
\dot{B}^{-1+\frac{d}{p}}_{p,2}$ with $\nabla\cdot
u_{0}=0$, $v_{0}, w_{0}\in
\dot{B}^{-2+\frac{d}{q}}_{q,2}$. Then there exists
$T>0$ such that the system \eqref{tranformed NSNPP} has a unique solution $(u,v,w)$ on $[0,T]$ satisfying
\begin{align*}
  u\in\widetilde{L}^{\infty}(0, T;
 \dot{B}^{-1+\frac{d}{p}}_{p,2})\cap \widetilde{L}^{2}(0, T;
  \dot{B}^{\frac{d}{p}}_{p,2}),\ \ \ \
  v,w\in  \widetilde{L}^{\infty}(0, T;
  \dot{B}^{-2+\frac{d}{q}}_{q,2})\cap \widetilde{L}^{2}(0, T;
\dot{B}^{-1+\frac{d}{q}}_{q,2}).
\end{align*}
Besides, there exists a positive constant $\varepsilon$ such that if
$$
\|u_{0}\|_{\dot{B}^{-1+\frac{d}{p}}_{p,2}}+\|(v_{0},
w_{0})\|_{\dot{B}^{-2+\frac{d}{q}}_{q,2}}
\leq\varepsilon,
$$
then the above assertion holds for $T=\infty$, i.e., the solution
$(u, v, w)$ is global.
\end{theorem}

To prove Theorem \ref{well-posedness in Besov spaces}, the crucial parts are the following bilinear estimates. In the sequel, we denote by $(d_{j})_{j\in\mathbb{Z}}$ a generic element of
$l^{2}(\mathbb{Z})$ such that $d_{j}\ge0$ and
$\sum_{j\in\mathbb{Z}}d_{j}^2=1$.
\begin{lemma}\label{convective term in NS}
Let $1\leq p<2d$. Then there exists a positive constant $C$ such that
\begin{align}\label{esimate of convective term}
   \|u\cdot\nabla
   u\|_{\widetilde{L}^{2}_{T}(\dot{B}^{-2+\frac{d}{p}}_{p,2})}\leq C\|u\|_{\widetilde{L}^{\infty}_{T}(\dot{B}^{-1+\frac{d}{p}}_{p,2})}
   \|u\|_{\widetilde{L}^{2}_{T}(\dot{B}^{\frac{d}{p}}_{p,2})}.
\end{align}
\end{lemma}
\begin{proof}
For the term $u\cdot \nabla u$, by using the product estimate \eqref{product estimates in Chemin-Lerner spaces1} and the interpolation inequality \eqref{interpolation in Chemin-Lerner spaces}, we  see that
\begin{align*}
   \|u\cdot\nabla u\|_{\widetilde{L}^{2}_{T}(\dot{B}^{-2+\frac{d}{p}}_{p,2})}&=\|\nabla\cdot (u\otimes u)\|_{\widetilde{L}^{2}_{T}(\dot{B}^{-2+\frac{d}{p}}_{p,2})}\leq C\|u\otimes u\|_{\widetilde{L}^{2}_{T}(\dot{B}^{-1+\frac{d}{p}}_{p,2})}\\
   &\leq C  \|u\|_{\widetilde{L}^{4}_{T}(\dot{B}^{-\frac{1}{2}+\frac{d}{p}}_{p,2})}^2 \leq C\|u\|_{\widetilde{L}^{\infty}_{T}(\dot{B}^{-1+\frac{d}{p}}_{p,2})}
   \|u\|_{\widetilde{L}^{2}_{T}(\dot{B}^{\frac{d}{p}}_{p,2})}.
\end{align*}
We complete the proof of Lemma \ref{convective term in NS}.
\end{proof}

\begin{lemma}\label{external force in NS}
Let $1\leq p<d$, $1\leq q\leq 2p$ and $\frac{1}{p}-\frac{1}{q}<\frac{1}{d}$. Then there exists a positive constant $C$ such that
\begin{align}\label{esimate of external force of u}
  \|v\nabla(-\Delta)^{-1}v\|_{\widetilde{L}^{2}_{T}(\dot{B}^{-2+\frac{d}{p}}_{p,2})}&\leq C   \|v\|_{\widetilde{L}^{2}_{T}(\dot{B}^{-1+\frac{d}{q}}_{q,2})}\|v\|_{\widetilde{L}^{\infty}_{T}(\dot{B}^{-2+\frac{d}{q}}_{q,2})}.
\end{align}
\end{lemma}
\begin{proof} We divide the proof into the following two cases.
\smallbreak

\textit{Case 1, $1\leq q\leq p$.}  In this case, based on the observation
\begin{equation}\label{symmetric structure of v}
    v\nabla(-\Delta)^{-1}v=\nabla\cdot\big(  \nabla(-\Delta)^{-1}v\otimes\nabla(-\Delta)^{-1}v-\frac{1}{2}\left|\nabla(-\Delta)^{-1}v\right|^2 I\big)
\end{equation}
and the embedding relation $\dot{B}^{-2+\frac{d}{q}}_{q,2}\hookrightarrow\dot{B}^{-2+\frac{d}{p}}_{p,2}$, we can treat
$ \nabla(-\Delta)^{-1}v$ the same as $u$ to get the desired inequality \eqref{esimate of external force of u}.
\smallbreak

\textit{Case 2, $1\leq p<q$.} This is a tricky case. We first regard the term $ v\nabla(-\Delta)^{-1}v$ as the particular case of the following nonlinear term with symmetric structure:
\begin{equation*}
    v\nabla(-\Delta)^{-1}w+w\nabla(-\Delta)^{-1}v.
\end{equation*}
Then we resort to Bony's paraproduct decomposition \eqref{Bony decomposition} to get
\begin{equation}\label{eq4.5}
    v\nabla(-\Delta)^{-1}w+w\nabla(-\Delta)^{-1}v:=I_{1}+I_{2}+I_{3},
\end{equation}
where
\begin{align*}
    I_{1}:&=\sum_{j'\in\mathbb{Z}}S_{j'-1}v\nabla(-\Delta)^{-1}\Delta_{j'}w
    +S_{j'-1}w\nabla(-\Delta)^{-1}\Delta_{j'}v,\\
    I_{2}:&=\sum_{j'\in\mathbb{Z}}\Delta_{j'}v\nabla(-\Delta)^{-1}S_{j'-1}w
    +\Delta_{j'}w\nabla(-\Delta)^{-1}S_{j'-1}v,\\
   I_{3}:&=\sum_{j'\in\mathbb{Z}}\Delta_{j'}v\nabla(-\Delta)^{-1}\widetilde{\Delta}_{j'}w
    +\Delta_{j'}w\nabla(-\Delta)^{-1}\widetilde{\Delta}_{j'}v.
\end{align*}
For $I_{1}$, it suffices to deal with the first term  $\sum_{j'\in\mathbb{Z}}S_{j'-1}v\nabla(-\Delta)^{-1}\Delta_{j'}w$ due to the second one can be done analogously. Using  Lemma \ref{Bernstein inequality}, we get that
\begin{align*}
   \|\Delta_{j}\sum_{j'\in\mathbb{Z}}S_{j'-1}v\nabla(-\Delta)^{-1}\Delta_{j'}w\|_{L^2_T(L^{p})} &\leq C \sum_{|j-j'|\leq 4}\|S_{j'-1}v\|_{L^\infty_T(L^{\frac{pq}{q-p}})}\|\nabla(-\Delta)^{-1}\Delta_{j'}w\|_{L^2_T(L^{q})}\nonumber\\
    &\leq C \sum_{|j-j'|\leq 4}\sum_{k\leq j'-2}2^{(\frac{2d}{q}-\frac{d}{p})k}\|\Delta_{k}v\|_{L^\infty_T(L^{q})}2^{-j'}\|\Delta_{j'}w\|_{L^2_T(L^{q})}\nonumber\\
  & \leq C \sum_{|j-j'|\leq 4}\sum_{k\leq j'-2}2^{(2+\frac{d}{q}-\frac{d}{p})k}2^{(-2+\frac{d}{q})k}\|\Delta_{k}v\|_{L^\infty_T(L^{q})}2^{-j'}\|\Delta_{j'}w\|_{L^2_T(L^{q})}\nonumber\\
  &\leq C \sum_{|j-j'|\leq 4}2^{(1+\frac{d}{q}-\frac{d}{p})j'}\|\Delta_{j'}w\|_{L^2_T(L^{q})}\|v\|_{\widetilde{L}^\infty_T(\dot{B}^{-2+\frac{d}{q}}_{q,2})}\nonumber\\
  &\leq C 2^{(2-\frac{d}{p})j}d_{j}\|w\|_{\widetilde{L}^2_T(\dot{B}^{-1+\frac{d}{q}}_{q,2})}\|v\|_{\widetilde{L}^\infty_T(\dot{B}^{-2+\frac{d}{q}}_{q,2})},
\end{align*}
which directly leads to
\begin{align}\label{eq4.6}
  \|\Delta_{j}I_{1}\|_{L^2_T(L^{p})}\leq C
       2^{(2-\frac{d}{p})j}d_{j}\Big(\|w\|_{\widetilde{L}^2_T(\dot{B}^{-1+\frac{d}{q}}_{q,2})}\|v\|_{\widetilde{L}^\infty_T(\dot{B}^{-2+\frac{d}{q}}_{q,2})}
       +\|v\|_{\widetilde{L}^2_T(\dot{B}^{-1+\frac{d}{q}}_{q,2})}\|w\|_{\widetilde{L}^\infty_T(\dot{B}^{-2+\frac{d}{q}}_{q,2})} \Big).
\end{align}
Similarly, for the first term of $I_{2}$,  we get
\begin{align*}
    \|\Delta_{j}\sum_{j'\in\mathbb{Z}}\Delta_{j'}v\nabla(-\Delta)^{-1}S_{j'-1}w\|_{L^2_T(L^{p})} &\leq C  \sum_{|j-j'|\leq 4}\|\Delta_{j'}v\|_{L^2_T(L^{q})}\|\nabla(-\Delta)^{-1}S_{j'-1}w\|_{L^\infty_T(L^{\frac{pq}{q-p}})}\nonumber\\
     & \leq C \sum_{|j-j'|\leq 4}\|\Delta_{j'}v\|_{L^2_T(L^{q})}\sum_{k\leq j'-2}2^{(-1+\frac{2d}{q}-\frac{d}{p})k}\|\Delta_{k}w\|_{L^\infty_T(L^{q})}\nonumber\\
  &\leq C  \sum_{|j-j'|\leq 4}\|\Delta_{j'}v\|_{L^2_T(L^{q})}\sum_{k\leq j'-2}2^{(1+\frac{d}{q}-\frac{d}{p})k}2^{(-2+\frac{d}{q})k}\|\Delta_{k}w\|_{L^\infty_T(L^{q})}\nonumber\\
  &\leq C \sum_{|j-j'|\leq 4}2^{(1+\frac{d}{q}-\frac{d}{p})j'}\|\Delta_{j'}v\|_{L^2_T(L^{q})}\|w\|_{\widetilde{L}^\infty_T(\dot{B}^{-2+\frac{d}{q}}_{q,2})}\nonumber\\
  &\leq C  2^{(2-\frac{d}{p})j}d_{j}\|v\|_{\widetilde{L}^2_T(\dot{B}^{-1+\frac{d}{q}}_{q,2})}\|w\|_{\widetilde{L}^\infty_T(\dot{B}^{-2+\frac{d}{q}}_{q,2})},
\end{align*}
which yields that
\begin{align}\label{eq4.7}
     \|\Delta_{j}I_{2}\|_{L^2_T(L^{p})}\leq C  2^{(2-\frac{d}{p})j}d_{j}\Big(\|v\|_{\widetilde{L}^2_T(\dot{B}^{-1+\frac{d}{q}}_{q,2})}\|w\|_{\widetilde{L}^\infty_T(\dot{B}^{-2+\frac{d}{q}}_{q,2})}
     +\|w\|_{\widetilde{L}^2_T(\dot{B}^{-1+\frac{d}{q}}_{q,2})}\|v\|_{\widetilde{L}^\infty_T(\dot{B}^{-2+\frac{d}{q}}_{q,2})}\Big).
\end{align}
Finally we tackle with the most tricky term $I_{3}$. As we observed in \cite{ZL17}, we can split $I_{3}$ into the following three terms for $m=1,2,3$:
\begin{equation}\label{eq4.8}
  I_{3}:=J_{1}+J_{2}+J_{3},
\end{equation}
where
\begin{align*}
    J_{1}:&=\sum_{j'\in\mathbb{Z}}(-\Delta)\Big{\{}\big{(}(-\Delta)^{-1}\Delta_{j'}v\big{)}\big{(}\partial_{m}(-\Delta)^{-1}\widetilde{\Delta}_{j'}w\big{)}\Big{\}},\\
    J_{2}:&=\sum_{j'\in\mathbb{Z}}2\nabla\cdot\Big{\{}\big{(}(-\Delta)^{-1}\Delta_{j'}v\big{)}\big{(}\partial_{m}\nabla(-\Delta)^{-1}\widetilde{\Delta}_{j'}w\big{)}\Big{\}},\\
    J_{3}:&=\sum_{j'\in\mathbb{Z}}\partial_{m}\Big{\{}\big{(}(-\Delta)^{-1}\Delta_{j'}v\big{)}\widetilde{\Delta}_{j'}w\Big{\}}.
\end{align*}
Since $J_{2}$ can be estimated similarly to $J_{3}$, we deal with $J_{1}$ and $J_{3}$
only. It follows from the condition $1\leq p<d$ that
\begin{align*}
    \|\Delta_{j}J_{1}\|_{L^2_T(L^{p})}&\leq C  2^{2j}\sum_{j'\geq j-N_{0}}
    \|(-\Delta)^{-1}\Delta_{j'}v\|_{L^\infty_T(L^{\frac{pq}{q-p}})}\|\partial_{m}(-\Delta)^{-1}\widetilde{\Delta}_{j'}w\|_{L^2_T(L^{q})}\nonumber\\
    &\leq C  2^{2j}\sum_{j'\geq j-N_{0}}2^{(-2+\frac{2d}{q}-\frac{d}{p})j'}
    \|\Delta_{j'}v\|_{L^\infty_T(L^{q})}2^{-j'}\|\widetilde{\Delta}_{j'}w\|_{L^2_T(L^{q})}\nonumber\\
    &\leq C  2^{2j}\sum_{j'\geq j-N_{0}}2^{-\frac{d}{p}j'}2^{(-2+\frac{d}{q})j'}
    \|\Delta_{j'}v\|_{L^\infty_T(L^{q})}2^{(-1+\frac{d}{q})j'}\|\widetilde{\Delta}_{j'}w\|_{L^2_T(L^{q})}\nonumber\\
    &\leq C  2^{(2-\frac{d}{p})j}d_{j}\|w\|_{\widetilde{L}^2_T(\dot{B}^{-1+\frac{d}{q}}_{q,2})}\|v\|_{\widetilde{L}^\infty_T(\dot{B}^{-2+\frac{d}{q}}_{q,2})},
\end{align*}
\begin{align*}
    \|\Delta_{j}J_{3}\|_{L^2_T(L^{p})}&\leq C  2^{j}\sum_{j'\geq j-N_{0}}
    \|(-\Delta)^{-1}\Delta_{j'}v\|_{L^\infty_T(L^{\frac{pq}{q-p}})}\|\widetilde{\Delta}_{j'}w\|_{L^2_T(L^{q})}\nonumber\\
   &\leq C  2^{j}\sum_{j'\geq j-N_{0}}2^{(1-\frac{d}{p})j'}2^{(-2+\frac{d}{q})j'}
    \|\Delta_{j'}v\|_{L^\infty_T(L^{q})}2^{(-1+\frac{d}{q})j'}\|\widetilde{\Delta}_{j'}w\|_{L^2_T(L^{q})}\nonumber\\
    &\leq C  2^{(2-\frac{d}{p})j}d_{j}\|w\|_{\widetilde{L}^2_T(\dot{B}^{-1+\frac{d}{q}}_{q,2})}\|v\|_{\widetilde{L}^\infty_T(\dot{B}^{-2+\frac{d}{q}}_{q,2})}.
\end{align*}
As a consequence,  we deduce from \eqref{eq4.8} that
\begin{align}\label{eq4.9}
      \|\Delta_{j} I_{3}\|_{L^2_T(L^{p})}\leq C
      2^{(2-\frac{d}{p})j}d_{j}\|w\|_{\widetilde{L}^2_T(\dot{B}^{-1+\frac{d}{q}}_{q,2})}\|v\|_{\widetilde{L}^\infty_T(\dot{B}^{-2+\frac{d}{q}}_{q,2})}.
\end{align}
Hence, plugging \eqref{eq4.6}, \eqref{eq4.7} and \eqref{eq4.9} into \eqref{eq4.5}, we obtain \eqref{esimate of external force of u}. We complete the proof of Lemma \ref{external force in NS}.
\end{proof}

\begin{lemma}\label{convective term of w}
Let $1\leq p,q<\infty$ and $\frac{1}{q}-\frac{1}{p}<\frac{3}{d}$,
$\frac{1}{p}+\frac{1}{q}>\frac{2}{d}$. Then there exists a positive constant $C$ such that
\begin{align}\label{esimate of convective term of w}
   \|u\cdot\nabla
   v\|_{\widetilde{L}^{2}_{T}(\dot{B}^{-3+\frac{d}{q}}_{q,2})}\leq C\Big( \|u\|_{\widetilde{L}^{\infty}_{T}(\dot{B}^{-1+\frac{d}{p}}_{p,2})}
   \|v\|_{\widetilde{L}^{2}_{T}(\dot{B}^{-1+\frac{d}{q}}_{q,2})}+\|u\|_{\widetilde{L}^{2}_{T}(\dot{B}^{\frac{d}{p}}_{p,2})}
   \|v\|_{\widetilde{L}^{\infty}_{T}(\dot{B}^{-2+\frac{d}{q}}_{q,2})}\Big).
\end{align}
\end{lemma}
\begin{proof}
Thanks to Bony's paraproduct decomposition \eqref{Bony decomposition}, we have
\begin{equation*}
   u\cdot\nabla v=T_{u}\nabla v+T_{\nabla v}u+R(u,\nabla v).
\end{equation*}
We first estimate the term $T_{u}\nabla v$.  Applying Lemma \ref{Bernstein inequality} gives us to
\begin{align}\label{convective term1}
  \|\Delta_{j}(T_{u}\nabla v)\|_{L^2_T(L^{q})}
  &\leq C  \sum_{|j'-j|\le4}2^{j'}\|S_{j'-1}u\|_{L^\infty_T(L^{\infty})}\|\Delta_{j'}v\|_{L^2_T(L^{q})}\nonumber\\
  &\leq C  \sum_{|j'-j|\le4}2^{j'}\sum_{k\le j'-2}2^{\frac{d}{p}k}\|\Delta_{k}u\|_{L^\infty_T(L^{p})}\|\Delta_{j'}v\|_{L^2_T(L^{q})}
  \nonumber\\
   &\leq C  \sum_{|j'-j|\le4}2^{2j'}\|\Delta_{j'}v\|_{L^2_T(L^{q})}\|u\|_{\widetilde{L}^\infty_T(\dot{B}^{-1+\frac{d}{p}}_{p,2})}\nonumber\\
   &\leq C  2^{(3-\frac{d}{q})j}d_{j}\|v\|_{\widetilde{L}^2_T(\dot{B}^{-1+\frac{d}{q}}_{q,2})}
 \|u\|_{\widetilde{L}^\infty_T(\dot{B}^{-1+\frac{d}{p}}_{p,2})}.
\end{align}
Next, for the term $T_{\nabla v}u$,  we treat it in the following two cases. In the case $1\leq q\leq p$,  we have
\begin{align}\label{convective term21}
  \|\Delta_{j}(T_{\nabla v}u)\|_{L^2_T(L^{q})}&\leq C
  \sum_{|j'-j|\le4}\|\Delta_{j}S_{j'-1}\nabla v\|_{L^\infty_T(L^{\frac{pq}{p-q}})}\|\Delta_{j'}u\|_{L^2_T(L^{p})}\nonumber\\
  &\leq C  \sum_{|j'-j|\le4}\sum_{k\le j'-2}2^{(1+\frac{d}{p})k}\|\Delta_{k} v\|_{L^\infty_T(L^{q})}
  \|\Delta_{j'}u\|_{L^2_T(L^{p})}
  \nonumber\\
   &\leq C  \sum_{|j'-j|\le4}\sum_{k\le j'-2}2^{(3+\frac{d}{p}-\frac{d}{q})k}2^{(-2+\frac{d}{q})k}\|\Delta_{k} v\|_{L^\infty_T(L^{q})}
  \|\Delta_{j'}u\|_{L^2_T(L^{p})}
   \nonumber\\
  &\leq C  2^{(3-\frac{d}{q})j}d_{j}\|v\|_{\widetilde{L}^\infty_T(\dot{B}^{-2+\frac{d}{q}}_{q,2})}
 \|u\|_{\widetilde{L}^2_T(\dot{B}^{\frac{d}{p}}_{p,2})};
\end{align}
while in the case  $1\leq p<q$,  we have
\begin{align}\label{convective term22}
  \|\Delta_{j}(T_{\nabla v}u)\|_{L^2_T(L^{q})}&\leq C  2^{d(\frac{1}{p}-\frac{1}{q})j}
  \sum_{|j'-j|\le4}\|S_{j'-1}\nabla v\Delta_{j'}u\|_{L^2_T(L^{p})}\nonumber\\
  &\leq C  2^{d(\frac{1}{p}-\frac{1}{q})j}\sum_{|j'-j|\le4}\sum_{k\le j'-2}2^{(1+\frac{d}{q})k}\|\Delta_{k} v\|_{L^\infty_T(L^{q})}
  \|\Delta_{j'}u\|_{L^2_T(L^{p})}
  \nonumber\\
   &\leq C  2^{d(\frac{1}{p}-\frac{1}{q})j}\sum_{|j'-j|\le4}\sum_{k\le j'-2}2^{3k}2^{(-2+\frac{d}{q})k}\|\Delta_{k} v\|_{L^\infty_T(L^{q})}
  \|\Delta_{j'}u\|_{L^2_T(L^{p})}
  \nonumber\\
  &\leq C  2^{(3-\frac{d}{q})j} d_{j}
  \|v\|_{\widetilde{L}^\infty_T(\dot{B}^{-2+\frac{d}{q}}_{q,2})}
 \|u\|_{\widetilde{L}^2_T(\dot{B}^{\frac{d}{p}}_{p,2})}.
\end{align}
Finally, we deal with the remaining term $R(u,\nabla v)$. In the case that $\frac{1}{p}+\frac{1}{q}\leq
1$, the condition $\frac{1}{p}+\frac{1}{q}> \frac{2}{d}$ implies that
\begin{align}\label{convective term31}
  \|\Delta_{j}R(u,\nabla v)\|_{L^2_T(L^{q})}
  &\leq C  2^{(1+\frac{d}{p})j}\sum_{j'\ge j-N_{0}}\|\Delta_{j'}u\|_{L^2_T(L^{p})}\|\widetilde{\Delta}_{j'}v\|_{L^\infty_T(L^{q})}\nonumber\\
  &\leq C  2^{(1+\frac{d}{p})j}\sum_{j'\ge j-N_{0}}2^{(2-\frac{d}{p}-\frac{d}{q})j'}2^{\frac{d}{p}j'}\|\Delta_{j'}u\|_{L^2_T(L^{p})}
  2^{(-2+\frac{d}{q})j'}\|\widetilde{\Delta}_{j'}v\|_{L^\infty_T(L^{q})}\nonumber\\
  &\leq C  2^{(3-\frac{d}{q})j}d_{j} \|v\|_{\widetilde{L}^\infty_T(\dot{B}^{-2+\frac{d}{q}}_{q,2})}
 \|u\|_{\widetilde{L}^2_T(\dot{B}^{\frac{d}{p}}_{p,2})}.
\end{align}
In the case that $\frac{1}{p}+\frac{1}{q}>1$, we find
$1<q'\le\infty$ such that $\frac{1}{q}+\frac{1}{q'}=1$, then
\begin{align}\label{convective term32}
  \|\Delta_{j}R(u,\nabla v)\|_{L^2_T(L^{q})}&\leq C
  2^{j+d(1-\frac{1}{q})j}\sum_{j'\ge j-N_{0}}\|\Delta_{j'}u\widetilde{\Delta}_{j'}v\|_{L^2_T(L^{1})}\nonumber\\
  &\leq C  2^{j+d(1-\frac{1}{q})j}\sum_{j'\ge j-N_{0}}\|\Delta_{j'}u\|_{L^2_T(L^{q'})}\|\widetilde{\Delta}_{j'}v\|_{L^\infty_T(L^{q})}\nonumber\\
  &\leq C  2^{j+d(1-\frac{1}{q})j}\sum_{j'\ge j-N_{0}}2^{-(d-2)j'}2^{\frac{d}{p}j'}\|\Delta_{j'}u\|_{L^2_T(L^{p})}
  2^{(-2+\frac{d}{q})j'}\|\widetilde{\Delta}_{j'}v\|_{L^\infty_T(L^{q})}\nonumber\\
  &\leq C  2^{(3-\frac{d}{q})j}d_{j}\|v\|_{\widetilde{L}^\infty_T(\dot{B}^{-2+\frac{d}{q}}_{q,2})}
 \|u\|_{\widetilde{L}^2_T(\dot{B}^{\frac{d}{p}}_{p,2})}.
\end{align}
Putting \eqref{convective term1}--\eqref{convective term32} together, one gets \eqref{esimate of convective term of w}. We complete the proof of Lemma \ref{convective term of w}.
\end{proof}

\begin{lemma}\label{drift term of v}
Let $1\leq q<d$. Then there exists a positive constant $C$ such that
\begin{align}\label{drift term of v1}
   \|w\nabla(-\Delta)^{-1}v\|_{\widetilde{L}^{2}_{T}(\dot{B}^{-2+\frac{d}{q}}_{q,2})}\leq C\Big(
   \|w\|_{\widetilde{L}^{\infty}_{T}(\dot{B}^{-2+\frac{d}{q}}_{q,2})}
   \|v\|_{\widetilde{L}^{2}_{T}(\dot{B}^{-1+\frac{d}{q}}_{q,2})}+\|v\|_{\widetilde{L}^{\infty}_{T}(\dot{B}^{-2+\frac{d}{q}}_{q,2})}
   \|w\|_{\widetilde{L}^{2}_{T}(\dot{B}^{-1+\frac{d}{q}}_{q,2})}\Big).
\end{align}
\end{lemma}
\begin{proof}
Thanks to Bony's paraproduct decomposition \eqref{Bony decomposition}, we obtain
\begin{equation*}
   w\nabla(-\Delta)^{-1}v=T_{w}\nabla(-\Delta)^{-1}v+T_{\nabla(-\Delta)^{-1}v}w+R(w,\nabla(-\Delta)^{-1}v).
\end{equation*}
Applying Lemma  \ref{Bernstein inequality}  and Lemma \ref{Properties}, we obtain
\begin{align}\label{drift term of v2}
  \|\Delta_{j}(T_{w}\nabla(-\Delta)^{-1}v)\|_{L^2_{T}(L^{q})}
  &\leq C \sum_{|j'-j|\le4}\|S_{j'-1}w\|_{L^\infty_T(L^{\infty})}\|\Delta_{j'}\nabla(-\Delta)^{-1}v\|_{L^2_T(L^{q})}\nonumber\\
  &\leq C \sum_{|j'-j|\le4}2^{-j'}\sum_{k\leq j'-2}2^{\frac{d}{q}k}\|\Delta_{k}w\|_{L^\infty_T(L^{q})}\|\Delta_{j'}v\|_{L^2_T(L^{q})}\nonumber\\
  &\leq C \sum_{|j'-j|\le4}2^{-j'}\sum_{k\leq j'-2}2^{2k}2^{(-2+\frac{d}{q})k}
  \|\Delta_{k}w\|_{L^\infty_T(L^{q})}
  \|\Delta_{j'}v\|_{L^2_T(L^{q})}\nonumber\\
  &\leq C 2^{(2-\frac{d}{q})j}d_{j}\|w\|_{\widetilde{L}^\infty_T(\dot{B}^{-2+\frac{d}{q}}_{q,2})}\|v\|_{\widetilde{L}^2_T(\dot{B}^{-1+\frac{d}{q}}_{q,2})},
\end{align}
\begin{align}\label{drift term of v3}
  \|\Delta_{j}(T_{\nabla(-\Delta)^{-1}v}w)\|_{L^2_T(L^{q})}
  &\leq C \sum_{|j'-j|\le4}\|S_{j'-1}\nabla(-\Delta)^{-1}v\|_{L^\infty_T(L^{\infty})}\|\Delta_{j'}w\|_{L^2_T(L^{q})}\nonumber\\
  &\leq C  \sum_{|j'-j|\le4}\sum_{k\leq j'-2}2^{(-1+\frac{d}{q})k}\|\Delta_{k}v\|_{L^\infty_T(L^{q})}\|\Delta_{j'}w\|_{L^2_T(L^{q})}\nonumber\\
  &\leq C  \sum_{|j'-j|\le4}\sum_{k\leq j'-2}2^{k}2^{(-2+\frac{d}{q})k}
  \|\Delta_{k}v\|_{L^\infty_T(L^{q})}
  \|\Delta_{j'}w\|_{L^2_T(L^{q})}\nonumber\\
  &\leq C  2^{(2-\frac{d}{q})j}d_{j}\|v\|_{\widetilde{L}^\infty_T(\dot{B}^{-2+\frac{d}{q}}_{q,2})}\|w\|_{\widetilde{L}^2_T(\dot{B}^{-1+\frac{d}{q}}_{q,2})}.
\end{align}
To bound the remaining term $R(w,\nabla(-\Delta)^{-1}v)$, in the case $1\leq q<2$, there exists $2<q'\le\infty$ such that
$\frac{1}{q}+\frac{1}{q'}=1$, and  using Lemma \ref{Bernstein inequality}, one gets
\begin{align}\label{drift term of v41}
  \|\Delta_{j}R(w,\nabla(-\Delta)^{-1}v)\|_{L^2_T(L^{q})}&\leq C
  2^{d(1-\frac{1}{q})j}\sum_{j'\ge j-N_{0}}\|\Delta_{j'}w\widetilde{\Delta}_{j'}\nabla(-\Delta)^{-1}v\|_{L^2_T(L^{1})}\nonumber\\
  &\leq C  2^{d(1-\frac{1}{q})j}\sum_{j'\ge j-N_{0}}\|\Delta_{j'}w\|_{L^\infty_T(L^{q'})}
 2^{-j'} \|\widetilde{\Delta}_{j'}v\|_{L^2_T(L^{q})}\nonumber\\
  &\leq C  2^{d(1-\frac{1}{q})j}\sum_{j'\ge j-N_{0}}2^{-(d-2)j'}2^{(-2+\frac{d}{q})j'}\|\Delta_{j'}w\|_{L^\infty_T(L^{q})}
  2^{(-1+\frac{d}{q})j'}\|\widetilde{\Delta}_{j'}v\|_{L^2_T(L^{q})}\nonumber\\
  &\leq C  2^{(2-\frac{d}{q})j}d_{j}\|w\|_{\widetilde{L}^\infty_T(\dot{B}^{-2+\frac{d}{q}}_{q,2})}\|v\|_{\widetilde{L}^2_T(\dot{B}^{-1+\frac{d}{q}}_{q,2})};
\end{align}
while in the case $2\leq
q<d$, one can treat it as
\begin{align}\label{drift term of v42}
  \|\Delta_{j}R(w,\nabla(-\Delta)^{-1}v)\|_{L^2_T(L^{q})}&\leq C
  2^{\frac{d}{q}j}\sum_{j'\ge j-N_{0}}\|\Delta_{j'}w\widetilde{\Delta}_{j'}\nabla(-\Delta)^{-1}v\|_{L^2_T(L^{\frac{q}{2}})}\nonumber\\
  &\leq C  2^{\frac{d}{q}j}\sum_{j'\ge j-N_{0}}\|\Delta_{j'}w\|_{L^\infty_T(L^{q})}2^{-j'}
  \|\widetilde{\Delta}_{j'}v\|_{L^2_T(L^{q})}\nonumber\\
  &\leq C  2^{\frac{d}{q}j}\sum_{j'\ge j-N_{0}}2^{(2-\frac{2d}{q})j'}2^{(-2+\frac{d}{q})j'}\|\Delta_{j'}w\|_{L^\infty_T(L^{q})}
  2^{(-1+\frac{d}{q})j'}\|\widetilde{\Delta}_{j'}v\|_{L^2_T(L^{q})}\nonumber\\
  &\leq C  2^{(2-\frac{d}{q})j}d_{j}\|w\|_{\widetilde{L}^\infty_T(\dot{B}^{-2+\frac{d}{q}}_{q,2})}\|v\|_{\widetilde{L}^2_T(\dot{B}^{-1+\frac{d}{q}}_{q,2})}.
\end{align}
Putting \eqref{drift term of v2}--\eqref{drift term of v42} together, one gets \eqref{drift term of v1}. We complete that the proof of
Lemma \ref{drift term of v}.
\end{proof}

\begin{lemma}\label{drift term of w}
Let $1\leq q<2d$.  Then there exists a positive constant $C$ such that
\begin{align}\label{drift term of w1}
\|\nabla\cdot(v\nabla(-\Delta)^{-1}v)\|_{\widetilde{L}^{2}_{T}(\dot{B}^{-3+\frac{d}{q}}_{q,2})}\leq C
   \|v\|_{\widetilde{L}^{\infty}_{T}(\dot{B}^{-2+\frac{d}{q}}_{q,2})}
   \|v\|_{\widetilde{L}^{2}_{T}(\dot{B}^{-1+\frac{d}{q}}_{q,2})}.
\end{align}
\end{lemma}
\begin{proof}
By using Lemma \ref{Properties}, \eqref{symmetric structure of v} and Lemma \ref{product estimates in Chemin-Lerner spaces}, we see that
\begin{align*}
\|\nabla\cdot(v\nabla(-\Delta)^{-1}v)\|_{\widetilde{L}^{2}_{T}(\dot{B}^{-3+\frac{d}{q}}_{q,2})}&\leq C
\|v\nabla(-\Delta)^{-1}v\|_{\widetilde{L}^{2}_{T}(\dot{B}^{-2+\frac{d}{q}}_{q,2})}\nonumber\\
&\leq C
\|\nabla(-\Delta)^{-1}v\otimes\nabla(-\Delta)^{-1}v-\frac{1}{2}\left|\nabla(-\Delta)^{-1}v\right|^2 I\|_{\widetilde{L}^{2}_{T}(\dot{B}^{-1+\frac{d}{q}}_{q,2})} \nonumber\\
&\leq C\|\nabla(-\Delta)^{-1}v\|_{\widetilde{L}^{4}_{T}(\dot{B}^{-\frac{1}{2}+\frac{d}{q}}_{q,2})}^2
\nonumber\\
&\leq C\|v\|_{\widetilde{L}^{4}_{T}(\dot{B}^{-\frac{3}{2}+\frac{d}{q}}_{q,2})}^2
\nonumber\\
&\leq C
   \|v\|_{\widetilde{L}^{\infty}_{T}(\dot{B}^{-2+\frac{d}{q}}_{q,2})}
   \|v\|_{\widetilde{L}^{2}_{T}(\dot{B}^{-1+\frac{d}{q}}_{q,2})}.
\end{align*}
We complete the proof of
Lemma \ref{drift term of w}.
\end{proof}
\smallbreak

Based on Lemma \ref{convective term in NS}--Lemma \ref{drift term of w}, we can exactly follow the same argument as \cite{ZZL15} to get that there exists $T>0$ such that the system \eqref{tranformed NSNPP} admits a unique solution $(u,v,w)$ on $[0,T]$ satisfying
\begin{align*}
  u\in\widetilde{L}^{\infty}(0, T;
 \dot{B}^{-1+\frac{d}{p}}_{p,2})\cap \widetilde{L}^{2}(0, T;
  \dot{B}^{\frac{d}{p}}_{p,2}),\ \ \ \
  v,w\in  \widetilde{L}^{\infty}(0, T;
  \dot{B}^{-2+\frac{d}{q}}_{q,2})\cap \widetilde{L}^{2}(0, T;
\dot{B}^{-1+\frac{d}{q}}_{q,2}).
\end{align*}
Moveover, if the initial data is small enough, then the solution
$(u, v, w)$ is global.   We complete the proof of Theorem \ref{well-posedness in Besov spaces}.
\vskip .3in

\noindent \textbf{Acknowledgements}
 The authors declared that they have no conflict of interest. This work is partially supported by the National Natural Science Foundation of China (no. 12361034) and the Natural Science Foundation of Shaanxi Province (no. 2026JC-YBMS-0014).
\vskip.2cm

\noindent \textbf{Data Availability Statement} No data was used for the research described in the article.


\end{document}